\documentclass[reqno,12pt]{amsart}

\usepackage{amsfonts}
\usepackage[comma, authoryear]{natbib}
\usepackage{enumerate}

\numberwithin{equation}{section}

\newtheorem{thm}{Theorem}
\newtheorem{lem}{Lemma}
\newtheorem{defn}{Definition}
\newtheorem{cor}{Corollary}
\newtheorem{prop}{Proposition}
\newtheorem{RM}{Remark}
\newtheorem{ex}{Example}

\def\disp{\displaystyle}

\def\mat#1#2{\left[\!\! \begin{array}{c} #1 \\ #2 \end{array} \!\!\right]}

\def\dy#1{\partial_{y_{#1}}}
\def\dx#1{\partial_{x_{#1}}}
\def\n{\nabla}

\begin{document}
\title[Explicit solutions of SDEs]{On explicit local solutions of It\^o diffusions}

\author[M. Kouritzin]{By Michael A.  Kouritzin}
\curraddr{Department of Mathematical and Statistical Sciences\\
University of Alberta  \\
Edmonton (Alberta)\\
Canada T6G 2G1}
\email{michaelk@ualberta.ca
\newline\indent {\it URL:} http://www.math.ualberta.ca/Kouritzin\_M.html}

\author[B. R\'emillard]{Bruno   R\'emillard}
\curraddr{Service de l'enseignement des m\'ethodes  quantitatives de gestion\\
HEC Montr\'eal \\
Montr\'eal (Qu\'ebec) \\
Canada H3T 2A7}
\email{bruno.remillard@hec.ca
\newline\indent {\it URL:} http://www.hec.ca/pages/bruno.remillard}
\thanks{Corresponding Author: Michael Kouritzin, michaelk@ualberta.ca}
\thanks{Partial funding in support of this work
was provided by the Natural Sciences and Engineering Research Council of
Canada.}
\renewcommand{\subjclassname}{\textup{2010} Mathematics Subject Classification}
\dedicatory{\large University of Alberta and HEC Montr\'eal} \maketitle

\begin{abstract}
Strong solutions of $p$-dimensional
stochastic differential equations $dX_t=b(X_t,t)dt+\sigma
(X_t,t)dW_t$, $X_s=x$ that can be represented locally in \emph{explicit simulation}
form $X_t=\phi\left(\int_s^t U_{s,u}dW_u,t\right)$ are considered.
Here; $W$ is a
multidimensional Brownian motion; $U,\phi $ are continuous
functions;
and
$b,\sigma,\phi $ are locally continuously differentiable.
The following three-way equivalence is established:
1) There exists such a representation from all
starting points $(x,s)$,
2) $U,\phi$ satisfies a set differential equations,
and 3) $b,\sigma$ satisfy commutation relations.
Next, construction theorems, based on a diffeomorphism between the solutions
$X$ and the strong solutions to a simpler It\^o integral equation, with a possible
deterministic component, are given.
Finally, motivating examples are provided and reference to its
importance in filtering and option pricing is given. 
\end{abstract}


\section{Introduction}\label{intro}
Inasmuch as computability can be of utmost importance, one often
confines selection of stochastic differential equation (SDE) models to
those facilitating calculation and simulation.
This is perhaps best
exemplified in mathematical finance, where the popularity of the
inaccurate Black-Scholes model is only justifiable through the
evaluation ease of the resulting derivative product formulae.
Indeed, \citet[p. 272]{Kunita:1984} writes in his notes on
SDEs that ``It is an important problem
in applications that we can compute the output from the input
explicitly''.
We shall call such solutions \emph{explicit solutions}.

Filtering applications (see \citet{Kouritzin:1998}), option pricing
applications (see companion paper \citet{Kouritzin16}) and pedagogical considerations initially prompted
our classifications of which It\^o processes $X_t^{x,s}$, starting
at $(x,s)$, are representable as a time-dependent function of a simple
stochastic integral $\phi^{x,s}\left(\int_s^t U_{s,u} dW_u,t\right)$.
However,
our determination of $\phi^{x,s}$, $U_{s,u}$ also facilitates an effective means
of calculation and simulation. To simulate, one merely needs to
compute the Gauss-Markov process $\int_s^tU_{s,u} dW_u$ at discrete
times and substitute these samples into $\phi^{x,s}$, which is often
known in closed form and otherwise is the solution of differential
equations that can be solved numerically a priori.
The idea is applied to strong solutions here and extended to
weak solutions of a popular financial model
in \citet{Kouritzin16}.
$\int_s^tU_{s,u} dW_u=\int_s^tU_{s,u}(X_u) dW_u$ can depend upon $X$ but
not in a way that will destroy its Gaussian distribution nor
make simulation difficult and our explicit solutions are diffusion
solutions for all starting points $(x,s)$.
This Explicit Solution Simulation is without (Euler or Milstein) bias and is extremely efficient, often
orders of magnitude faster than Euler or Milstein
methods when our method is applicable and high accuracy is desired
(see \citet{Kouritzin16}).
Our representations also make properties of certain
stochastic differential equations readily discernible and simplifies
some filtering calculations. Finally, as demonstrated in
\citet[Proposition 5.2.24]{Karatzas/Shreve:1987}, explicit solutions
can be useful in establishing convergence for solutions of
stochastic differential equations.

\citet{Doss:1977} and \citet{Sussmann:1978} were apparently the
first to solve stochastic differential equations through use of
differential equations. In the multidimensional setting, Doss
imposed the Abelian condition on the Lie algebra generated by the
vector fields of coefficients and showed, in this case, that strong
solutions, $X_t^{x}$, of Fisk-Stratonovich equations
are representable as $X_t^{x}=\rho (\Phi (x,W_{\cdot })_t,W_t)$,
for some continuous $\rho $, $\Phi $ solving differential equations.
Under the restriction of $C^\infty $ coefficients,
\citet{Yamato:1979}  extended the work of Doss by dispensing with
the Abelian assumption in favour of less restrictive $q$ step
nilpotency, whilst also introducing a simpler form for his explicit
solutions $X_t^{x}=u(x,t,(W_t^I)_{I\in F})$. Here, $u$ solves a
differential equation, and $(W_t^I)_{I\in F}$ are iterated
Stratonovich integrals with integrands and  integrators selected
from  $\left( t,W_t^1,...,W_t^d\right)$. Another  substantial work
on explicit solutions to stochastic differential equations is due to
\citet{Kunita:1984}[Section III.3]. He considers representing
solutions to time-homogeneous Fisk-Stratonovich equations via flows
generated by the coefficients of the equation under a commutative
condition similar to ours, and, more generally, under solvability of
the underlying Lie algebra. Kunita's work therefore generalizes
\citet{Yamato:1979}. Perhaps, the two most distinguishing features
of our work are: We allow time-dependent coefficients and utilize a
different representation that is very useful in simulation and
other applications (see e.g.\ \cite{Kouritzin:1998}, \cite{Kouritzin16}). 
We compare our results to \citet{Yamato:1979} and \citet{Kunita:1984} in
Section \ref{kun}.

In order to describe our method, we mention that the hitherto rather
ad hoc, state-space diffeomorphism mapping method has been used to
construct solutions to interesting stochastic differential equations
from solutions to simpler ones. 
The idea of this method is to change
the infinitesimal generator $L$ of a simple It\^o process to the
generator $\mathcal L$ corresponding to a more complicated It\^o process via ${
\mathcal L}f(x)= \{L(f\circ \Lambda^{-1} )\}\circ \Lambda(x)$. 
This corresponds to using It\^o's formula on $X_t = \Lambda^{-1}(\xi
_t)$, where
$\xi$ is a diffusion process with infinitesimal generator $L$.
For related
examples, we refer the reader  to the problems in
\citet{Friedman:2006}[page 126] or \citet{Ethier/Kurtz:1986}[page
303]. 

Motivated by applications in filtering, \citet{Kouritzin/Li:2000}
and \citet{Kouritzin:2000} used differential equation methods to
study: ``When can global, time-dependent diffeomorphisms be used to
construct solutions to It\^o equations?", ``What scalar It\^o
equations can be solved via diffeomorphisms?", and ``How can one
construct these diffeomorphisms?". They considered scalar solutions
in an open interval  $D$
 to the time-homogeneous stochastic differential equation
\begin{equation} \label{SDE1}
dX_t=b(X_t)dt+\sigma (X_t)dW_t,\ X_0=x,
\end{equation}
which are of the form $\phi^{x} \left(\int_0^tU_udW_u,t\right)$,
and showed that all nonsingular solutions of this form were actually
(time-dependent) diffeomorphisms $\Lambda^{-1} _t(\xi _t)$ with $\xi $
satisfying
$$
 d\xi _t=( \chi -\kappa \xi_t)dt+dW_t,\ \xi _0 =\Lambda_0(x).
$$
Nonsingular in this scalar case was interpreted as finiteness of $
\int_\lambda ^y\sigma ^{-1}(x)dx$ for some fixed point $\lambda $
and all $ y\in D$. 
(Their methods involve non-stochastic differential equations
that can continue to hold in the singular situations when global
diffeomorphisms fail.)

For our current work, we suppose henceforth that
$D\subset \mathbb R^p$ is a bounded convex domain, $T>0,$
and define
$$
D_T=\left\{\begin{array}{ll} D& \mbox{ if } \sigma,b \mbox{ do not depend on t}\\
D\times [0,T)& \mbox{ if either do}
\end{array} \right.
$$
so $(x,s)\in D_T$ means $x\in D$ when $D_T=D$.
Then, we resolve the question: ``When can
we explicitly solve vector-valued It\^o equations
\begin{equation} \label{VSDE1}
dX_t=b(X_t,t)dt+\sigma (X_t,t)dW_t,\ X_s=x,
\end{equation}
with the dimensions of $X_t, W_t$ being $p,d$ respectively, through
representations of the form $X_t^{x,s}= \phi^{x,s}\left(\int_s^t
U_{s,u} dW_u,t\right)$?''. This question is more precisely broken into
two separate important questions: ``For which $\sigma$ and $b$ does
such a strong-local-solution representation exists?" and ``What
conditions are required on $\phi$ and $U$ for such representations
with $\int_s^tU_{s,u} dW_u=\int_s^tU_{s,u}(X_u) dW_u$ still being Gauss-Markov?"
Equivalently, we consider ``When can the solutions to the
Fisk-Stratonovich equation
\begin{equation}\label{Strato}
dX_t^{x}=h(X_t^{x},t)dt+\sigma (X_t^{x},t)\bullet dW_t,
\end{equation}
with
\begin{equation}
h=b-{\frac 12}\sum_{j=1}^d\{\nabla_\varphi \sigma _j\}\sigma _j\mbox{ on }D_T
\label{bh}
\end{equation}
and $\sigma _j$ denoting the $j^{\rm th}$ column of the matrix $\sigma$, be
locally represented in this manner?" It follows from, for example,
\citet{Kunita:1984}[p. 239] that the unique local solutions to these
(\ref{VSDE1}) and (\ref{Strato}) are equal if (\ref{bh}) holds and
$\sigma $ is twice continuously differentiable or satisfies the
Fisk-Stratonovich acceptable condition in $D$, the latter being
discussed in \citet{Protter:2004}[Chapter 5]. We work with It\^o
equations to avoid these stronger assumptions on $\sigma $ but still
relate $b$ and $h$ through (\ref{bh}). Also, to obtain simple,
concrete necessary and sufficient conditions for such a
representation, we consider all solutions starting from each
$(x,s)\in D_T$. Actually, assuming natural regularity conditions
and using differential form techniques, we obtain very satisfying
answers to these question by showing the equivalence of the
following three conditions: 1) The SDEs (\ref{VSDE1}) have our
local-solution-representations for all starting points $(x,s)\in
D_T$. 2) The representation pair $\phi^{x,s}, U_{s,t}$ satisfy a
system of differential equations. 3) The SDE coefficients
$\sigma$ and $h$ satisfy simple commutator conditions.
In the process of establishing this three-way equivalence, we also answer
the question ``When is (\ref{VSDE1}) locally diffeomorphic to an
SDE with a simple diffusion coefficient?" i.e.\  ``When will it have
a representation as in (\ref{ODE1},\ref{SDE3}) to follow?".
It turns out that this representation facilitates explicit weak solution
of the important financial Heston model as is shown in 
\cite{Kouritzin16}.

Given precise conditions of when an It\^{o} equation has such a
representation, the next natural questions we answer are: ``What form
do the solutions have?" and ``How do you construct such solutions?"
In order to include as many interesting examples as
possible we will only require \emph{local} representation $X_t^{x,s}=\phi^{x,s}
\left(\int_s^tU_{s,u} dW_u,t\right)$ and allow $\sigma $ to have rank
less than $\min(p,d)$. The first opportunity borne out of allowing
the rank of $\sigma (x)$ to be less than $p$ is the ability to
handle time-dependent coefficients, treating time as an extra state.
The second advantage from allowing lesser rank than $\min(p,d)$ is
the extra richness afforded by appending a deterministic equation
into the diffeomorphism solution.
A third, important benefit of this general rank condition is
the possibility of producing explicit \emph{weak} solutions to
SDEs where no explicit strong solution exists (see \cite{Kouritzin16}).
In our construction results, we show that $\phi$ is constructed via
a time-dependent diffeomorphism $\Lambda_t$, which in turn is defined in terms of $\sigma$.
The diffeomorphism separates a representable SDEs into deterministic and stochastic differential equations:
$\Lambda_t(X_t)=(\overline X_t,\widetilde X_t)$, where
$\widetilde X_t \in \mathbb R^{p-r}$ is deterministic and satisfies
the differential equation
\begin{equation}\label{ODE1}
	\frac{d}{dt}\widetilde X_t  = \widetilde h(\widetilde X_t,t),
\end{equation}
while $\overline X_t$ is a
Gauss-Markov process satisfying
\begin{equation}\label{SDE3}
d\overline X_t = (\overline \theta(\widetilde X_t,t)+\overline \beta(\widetilde X_t,t)\overline X_t) dt+ \left(I_r\Big{|}\; \overline \kappa(\widetilde X_t,t)\right)dW_t.
\end{equation}
$\overline \kappa$ is determined (within an equivalence class) by $\sigma$ while 
$\overline \theta$, $\widetilde h$ and 
and $\overline \beta$ can be anything (subject to dimensional and differentiability regularity conditions).
These parameters allow us to handle a whole class of \emph{nonlinear} drift coefficients $b$ for a given $\sigma$
in the SDE (\ref{VSDE1}) for $X_t=\Lambda_t^{-1}(\overline X_t,\widetilde X_t)$.

In the next section, we introduce notation and state the main existence
results. In Section 3, we build off of these existence results to give our
construction results, illustrated with simple applications.
We compare our work to prior work of Yamato and Kunita in Section 4.
The proofs of all main results are postponed to Section 5.


\section{Notation and Existence Results} \label{main}

Let $(W_t)_{t\ge 0}$ be a standard $d$-dimensional Brownian motion
with respect to filtration $\{{ \mathcal F}_t\}_{t\ge 0}$ satisfying
the usual hypotheses on a complete probability space $(\Omega,{ \mathcal F},P)$.
We will use $\phi$ to denote a representation function and $x$ to denote a starting point as in the introduction.
On the other hand, $\varphi$ will denote a variable with the same dimension $p$ as $\phi$ and $x$.

For functions of time or paths of a stochastic process, we use $Z_t$ and $Z(t)$ interchangeably.
For a matrix $V$, $V_j$ will denote its $j^{\rm th}$ column vector and $V_{i,j}$ the
$i^{\rm th}$ element of this $j^{\rm th}$ column.

$B_z(\delta)$ denotes an open Euclidean ball centered at $z$ with radius $\delta>0$.
Suppose $m,r\in \mathbb N$, $O\subset \mathbb R^m$ is open and $I\subset [0,T)$ is an interval.
Then, $C(I)$ is the continuous functions on $I$ and $C^r(O)$ denotes the continuous functions whose partial derivatives up to order $r$
exist and are continuous on $O$.
Moreover, $C^{r,1}(O\times I)$ denotes the continuous functions $g(\varphi,t)$ whose mixed
partial derivatives in $\varphi\in O$ up to order $r$
and in $t\in I$ up to order $1$ all exist and are continuous functions on $O\times I$.
$C^{1}(O\times I)=C^{1,0}(O\times I)\cap C^{0,1}(O\times I)$.
(We only require one-sided derivatives in time to exist at interval endpoints.)
For such functions of both $\varphi$ and $t$,
$\nabla_\varphi g $ is the Jacobian matrix
of vector function $g$, that is $\left( \nabla_\varphi g \right)
_{i,j}=\partial _{\varphi_j}g_i$, while $\nabla g$ will include the
time derivative as the last column.

The purpose of our representations is to simulate a large class of processes
in an efficient manner, which leads to a dilemma.
We would like to allow $U_{s,t}$ to depend upon $X^{x,s}$ for generality
but not in a way that would destroy the ease of simulation.
Our approach to this dilemma is to allow $U_{s,t}$ to act as an operator on the functions 
$\phi^{x,s }(y_u,u)\big|_{u\in[s,t]}$ but then impose the condition
that the result $U_{s,t}\phi^{x,s }(y_\cdot,\cdot)$ can not depend upon $y$.
As we will expose below, this basically allows $U_{s,t}$ to depend upon
some hidden deterministic part of $X$ but not the purely stochastic part, saving
the Gaussian nature of
\begin{equation}\label{Ydefn}
Y^s_t=\int_s^t U_{s,u}\phi(Y_\cdot^s,\cdot)dW(u)=\int_s^t U_{s,u}\phi(0,\cdot)dW(u)
\end{equation}
so it can be computed off-line, which is the point of this work.
Then, $\phi$ must be differentiable enough to apply
It\^{o}'s formula and allow room for random process $Y^{s}_t$ to move.
Finally, we want $U_{s,t}\phi$ to satisfy some type of simple state
equation so it is easy to compute.
The precise regularity conditions for potential representations 
$X_t^{x,s}= \phi^{x,s}\left(Y^s_t,t\right)$ now follow:
\begin{itemize}
\item[$\mathcal{C}_1$: ] For each $(x,s)\in D_T$, there is  a $t_0=t_0^{x,s}>s$ and
a convex neighbourhood ${\mathcal N}^{x,s}\subset \mathbb R^d$ of $0$
such that $\phi^{x,s } \in C^{2,1}({\mathcal N}^{x,s}\times [s,t_0); \mathbb R^p)$ and
$t\rightarrow U_{s,t}\phi^{x,s }(y_\cdot,\cdot) \in C^{1}( [s,t_0);\mathbb R^{d\times d})$.
\item[$\mathcal{C}_2$: ] $\phi ^{x,s},U_{s,t}$ start correctly
\begin{equation}
	\phi ^{x,s}(0,s)  =  x,\ \ U_{s,s}\phi^{x,s}(0,s)=I_d\ \forall (x,s)\in D_T. \label{initphi} 
\end{equation}
\item[$\mathcal{C}_3$: ] $U_{s,t}\phi^{x,s}$ is non-singular on ${\mathcal N}^{x,s}\times [s,t_0)$ (with matrix inverse denoted $U^{-1}_{s,t}\phi^{x,s}$) and satisfies 
\begin{equation}\label{Urestrict}
U_{s,t}\phi^{x,s }(y_u,u)=U_{s,t}\phi^{x,s }(0,u)
\end{equation}
as well as
\begin{equation}\label{Ugroup}
\ \ \ \ \ \ \ \ \ U^{-1}_{s,t}\phi^{x,s}(y_t,t)\frac{d}{dt}U_{s,t}\phi^{x,s}(y_u,u)\big|_{u=t}
=\frac{d}{dt}U_{u,t}\phi^{\phi^{x,s}(y_u,u),u}(y_u,u)\big|_{u=t}.
\end{equation}
\end{itemize}
Then, (\ref{initphi},\ref{Ugroup}) imply
\begin{equation}\label{Ugroup1}
U^{-1}_{s,t}\phi^{x,s}(y_t,t)\frac{d}{dt}U_{s,t}\phi^{x,s}(y_u,u)\big|_{u=t}
=U^{-1}_{t,t}\phi^{\phi^{x,s}(y_t,t),t}\frac{d}{dt}U_{u,t}\phi^{\phi^{x,s}(y_u,u),u}\big|_{u=t}
\end{equation}
and therefore that $U$ is a (two parameter) semigroup.
We use (\ref{Urestrict}) to economize the notation $U_{s,t}\phi^{x,s }(y_\cdot,\cdot)$
to $U_{s,t}\phi^{x,s}$.


Now, define the $\mathcal{F}_t$-stopping time
\[
\tau^{x,s} = \min\left( t_0^{x,s},\inf\{t > s: \;
Y^{s}_t\notin{\mathcal N}^{x,s} \mbox{ or } (\phi^{x,s}(Y^{s}_t,t),t) \not \in D_T
\}\right)
\]
and let 
\begin{equation} \label{Rdefn}
\mathcal R^{x,s}=\mathop{\cup}\limits_{t\ge 0}\left\{(y,t):P((Y^{s}_t,t)\in B_{(y,t)}(\delta),t\le\tau^{x,s} )>0\ \forall \delta>0\right\}.
\end{equation}
There is structure that can be imposed upon $\phi,U$ that will turn out to be equivalent to the existence of our explicit strong local solutions.
\begin{defn}\label{RepPair}
An $(x,s,\sigma,h)$-representation is a pair $\phi^{x,s},U_{s,t}$ satisfying ($\mathcal{C}_1, \mathcal{C}_2, \mathcal{C}_3$) such that the following system of differential equations:
\begin{eqnarray}
\nabla _y\phi^{x,s}(y,t)& = & \sigma (\phi^{x,s}(y,t),t) U^{-1}_{s,t}\phi^{x,s}, \label{grad} \\
\partial _t\phi^{x,s}(y,t) & = & h(\phi^{x,s}(y,t),t) \label{dtphi}
\end{eqnarray}
hold for all $(y,t)\in{\mathcal R}^{x,s}$ and $\partial_s \nabla_y \phi^{x,s}(0,s)$,
$\partial_s \partial_t \phi^{x,s}(0,s)$, $\partial_{x_i} \nabla_y \phi^{x,s}(0,s)$
and $\partial_{x_i} \partial_t \phi^{x,s}(0,s)$ exist as continuous functions.
Here and below, $\partial_t \phi^{x,s}(0,s)$ means $\partial_t \phi^{x,s}(0,t)\big|_{t=s}$. 
\end{defn}

Now, our explicit solutions are:
\begin{equation} \label{rep}
X^{x,s}_t=\phi (Y_t,t)=\phi^{x,s}(Y^{s}_t,t) \mbox{ on }
[s,\tau^{x,s}).
\end{equation}
Our first main result establishes two necessary and sufficient conditions
for all $X^{x,s}$, defined in (\ref{rep}), to be strong local solutions to
\begin{equation} \label{SDE4}
 dX_t=b(X_t,t)dt+\sigma (X_t,t)dW_t,\qquad  X_s=x
\end{equation}
on $[s,\tau^{x,s})$.
The function $h$ is
always related to $b$ through (\ref{bh}) and $U_{s,t}\phi^{x,s}$ comes into the necessary and sufficient commutator conditions through generator
\begin{equation} \label{SecondA}
A(x,s)=\frac{d}{dt}  U_{s,t}\phi^{x,s}\big|_{t=s}.
\end{equation}
It follows from (\ref{Urestrict}) that $A$ does not depend upon $y$.


\medskip
\begin{thm}\label{r1}
The following are
equivalent:
\begin{enumerate}
\item[a)] 
$\sigma\in C^1(D_T;\mathbb R^{p\times d})$, $h\in C^1(D_T;\mathbb R^p)$, 
there is a unique strong solution to (\ref{SDE4}) on $[s,\tau^{x,s})$ 
for each $(x,s)\in D_T$, and this solution has explicit form 
$\phi^{x,s}(Y_t^s,t)$ with $\phi^{x,s},U_{s,t}$ satisfying $\mathcal{C}_1,\mathcal{C}_2,\mathcal{C}_3$.
\item[b)] There is a $(x,s,\sigma,h)$-representation $\phi^{x,s},U_{s,t}$ for each $(x,s)\in D_T$.
\item[c)] 
$\sigma\in C^1(D_T;\mathbb R^{p\times d})$, $h\in C^1(D_T;\mathbb R^p)$ and 
the following commutator conditions hold on $D_T$:
\begin{equation}\label{braket2}
(\nabla_\varphi \sigma _k)\sigma _j=(\nabla_\varphi \sigma _j)\sigma _k,
\mbox{ for all } j,k\in \{1,...,d\},
\end{equation}
\begin{equation} \label{h}
(\nabla_\varphi h)\sigma _j
=(\nabla_\varphi \sigma _j)h+\partial_t\sigma_j-\sigma A_j,\mbox{ for all } 1\le j\le d.
\end{equation}
\end{enumerate}
\end{thm}

\begin{RM}\label{timeinvcase}
Theorem \ref{r1} simplifies in the time-invariant $h,\sigma$ coefficient case.
Clearly, one only needs to check the commutator conditions on $D$ versus $D_T$.
However, the second commutator condition actually changes in form to:
\begin{equation} \label{h-timeinv}
(\nabla_\varphi h)\sigma _j
-(\nabla_\varphi \sigma _j)h=\sigma B_j,\mbox{ for all } 1\le j\le d,
\end{equation}
where $B(\varphi)=-A(\varphi,0)$.
Indeed, the left hand side of (\ref{h-timeinv}) does not depend on time
so the right side can not either.
\end{RM}

\begin{RM}
Theorem \ref{r1} also simplifies when $d=1$, which corresponds to appending a deterministic equation and
allowing time dependence to the case considered in \citet{Kouritzin:2000}.
In this $d=1$ case, (\ref{braket2}) is automatically true and (\ref{h}) becomes
\begin{equation} \label{hd1}
(\nabla_\varphi h)\sigma 
=(\nabla_\varphi \sigma )h+\partial_t\sigma-\sigma A.
\end{equation}
\end{RM}
Often, we are interested in establishing the representation for a given
stochastic differential equation.
In this case, the commutator conditions can be used quickly to determine if such
a representation is possible.
The easiest way to ensure (\ref{braket2}) is to have each column a constant multiple of another
$\sigma_j=c_j\sigma_1$ for all $j$ say.
However, there are other possibilities.
\begin{ex}
Let $p=d=2$ and $D\subset\mathbb R$ be a domain.
Suppose $a,e,f,g,m,n$ are $C^2(D)$-functions and our Fisk-Stratonovich equation has time-invariant coefficients:
\begin{equation}
h(\varphi_1,\varphi_2)=\left(\begin{matrix}f\left(\varphi_{1}\right)g\left(\varphi_{2}\right)\\m\left(\varphi_{1}\right)n\left(\varphi_{2}\right)\end{matrix}\right), \ 
\sigma \left(\varphi_{1},\varphi_{2}\right)=
\left(\begin{matrix}a\left(\varphi_{1}\right)&0\\e\left(\varphi_{2}\right)&e\left(\varphi_{2}\right)\end{matrix}\right).
\end{equation}
Moreover, suppose $a(\varphi_1)$ and $e(\varphi_2)$ are never $0$.
Then, $\sigma$ is always non-singular and it follows by (\ref{grad}) 
as well as the mean value theorem that for any $u\in[s,t]$
\[
\phi^{x,s}(y,u)-\phi^{x,s}(\hat y,u)=\sigma(\phi^{x,s}( y^*,u))
U^{-1}_{s,u}\phi^{x,s}\cdot(y-\hat y)
\]
with $y^*\in\mathcal N^{x,s}$ for $y,\hat y\in\mathcal N^{x,s}$ and any possible representation $\phi^{x,s},U_{s,t}$.
Hence, $\phi^{x,s}(y,u)=\phi^{x,s}(\hat y,u)\leftrightarrow y=\hat y$.
Therefore, it follows from (\ref{Urestrict}) that 
$U_{s,u}$ can not depend upon $\phi^{x,s}( y,u)$ 
for any $u\in[s,t]$ and 
$B$ in (\ref{h-timeinv}) is constant by (\ref{SecondA}).
Now,
\begin{equation}
\nabla_\varphi h =\left(\begin{matrix}f'(\varphi_1)g(\varphi_2)&f(\varphi_1)g'(\varphi_2)\\m'(\varphi_1)n(\varphi_2)&m(\varphi_1)n'(\varphi_2)\end{matrix}\right)
\end{equation}
and
\begin{equation}
\nabla_\varphi\sigma_2 =\left(\begin{matrix}0&0\\0&e'\left(\varphi_{2}\right)\end{matrix}\right),\ \nabla_\varphi\sigma_1 =\left(\begin{matrix}a'(\varphi_1)&0\\0&e'\left(\varphi_{2}\right)\end{matrix}\right)
\end{equation}
so the first commutator condition (\ref{braket2}) is fine since
\begin{equation}
\nabla_\varphi\sigma_1\sigma_2 =\left(\begin{matrix}0\\e'\left(\varphi_{2}\right)e(\varphi_2)\end{matrix}\right)=\nabla_\varphi\sigma_2 \sigma_1.
\end{equation}
Moreover,
\begin{equation}
\nabla_\varphi h\sigma_2-\nabla_\varphi \sigma_2 h =\left(\begin{matrix}e(\varphi_2)f(\varphi_1)g'(\varphi_2)\\m(\varphi_1)(e(\varphi_2)n'(\varphi_2)-e'(\varphi_2)n(\varphi_2))\end{matrix}\right)
\end{equation}
and
\begin{equation}
\nabla_\varphi h\sigma_1-\nabla_\varphi \sigma_1 h =\left(\begin{matrix}af'g+efg'-a'fg\\am'n+emn'-e'mn\end{matrix}\right).
\end{equation}
On the other hand, denoting $B=\left(\begin{matrix}b_{11}&b_{12}\\b_{21}&b_{22}\end{matrix}\right)$, we have
\begin{equation}
\sigma B =\left(\begin{matrix}ab_{11}&ab_{12}\\eb_{11}+eb_{21}&eb_{12}+eb_{22}\end{matrix}\right).
\end{equation}
Hence, by (\ref{h-timeinv}) there is an explicit solution if and only if
\begin{equation}
\!\!\!\left(\begin{matrix}af'g+efg'-a'fg&efg'\\am'n+emn'-e'mn&m(en'-e'n)\end{matrix}\right)=\left(\begin{matrix}ab_{11}&ab_{12}\\eb_{11}+eb_{21}&eb_{12}+eb_{22}\end{matrix}\right)
\end{equation}
for constants $b_{11},b_{12},b_{21},b_{22}$.
If $f=c_1a$, $n=c_2e$, $eg'=c_3$ and $m'a=c_4$ for some constants $c_1,c_2,c_3,c_4$, then it is easy to show that this condition
is met with $b_{22}=-c_1c_3$, $b_{21}=c_2c_4-c_1c_3$ and $b_{11}=b_{12}=c_1c_3$ so the representation holds for
\begin{equation}
h(\varphi_1,\varphi_2)=\left(\begin{matrix}\alpha\frac{g\left(\varphi_{2}\right)}{m'\left(\varphi_{1}\right)}\\\beta \frac{m\left(\varphi_{1}\right)}{g'\left(\varphi_{2}\right)}\end{matrix}\right), \ 
\sigma \left(\varphi_{1},\varphi_{2}\right)=\left(\begin{matrix}\frac{\gamma}{m'\left(\varphi_{1}\right)}&0\\\frac{\delta}{g'\left(\varphi_{2}\right)}&\frac{\delta}{g'\left(\varphi_{2}\right)}\end{matrix}\right),
\end{equation}
where $\alpha=c_1c_4,\beta=c_2c_3,\gamma=c_4,\delta=c_3$ are any constants and $g,m$ are $C^2$-functions with $\frac{1}{m'\left(\varphi_{1}\right)},\frac{1}{g'\left(\varphi_{1}\right)}\in C^1(D)$.
\end{ex}
\begin{ex}
In a similar manner, it follows that
\begin{equation}
h(\varphi_1,\varphi_2)=\left(\begin{matrix}\alpha\frac{g\left(\varphi_{2}\right)}{m'\left(\varphi_{1}\right)}\\\beta \frac{m\left(\varphi_{1}\right)}{g'\left(\varphi_{2}\right)}\end{matrix}\right), \ 
\sigma \left(\varphi_{1},\varphi_{2}\right)=\left(\begin{matrix}\frac{\gamma}{m'\left(\varphi_{1}\right)}&0\\0&\frac{\delta}{g'\left(\varphi_{2}\right)}\end{matrix}\right),
\end{equation}
for any constants $\alpha,\beta,\gamma,\delta$, also has a representation.
\end{ex}
There was significant work done in the previous examples and we still did not have
the representation functions.
The next example is the key to solving for complete representations and will be used in the following section.
\begin{ex}\label{CoreEG}
Suppose
$\sigma(\varphi,t) = \left(\begin{array}{cc} I_r & \overline \kappa(\varphi,t)\\
0 & 0 \end{array}\right) \in \mathbb R^{p\times d}$ satisfies (\ref{braket2}).
We will find the possible $h,b$ satisfying (\ref{h}) and the corresponding representations $U_{s,t},\phi^{x,s}$ by Theorem \ref{r1}.\\
{\bf Notation:}
As always, $\varphi$ is a variable and $\phi$ is the representation function.
Further, let $\overline{x}=(x_1,...,x_r)$, $\widetilde{x}=(x_{r+1},...,x_d)$,
$\overline{\varphi}=(\varphi_1,...,\varphi_r)$, $\widetilde{\varphi}=(\varphi_{r+1},...,\varphi_d)$,
$\widetilde D = \{\widetilde \varphi:(\overline \varphi,\widetilde \varphi) \in D\ \text{for some}\ \overline \varphi\}$,
$\widetilde D_T =\widetilde D \times[0,T)$,
\begin{equation}
\phi^{x,s}(y,t)=\left(\begin{array}{c}\overline\phi^{x,s}(y,t)\\\widetilde\phi^{x,s}(y,t)\end{array}\right),\ h=\left(\begin{array}{c}\overline h\\\widetilde h\end{array}\right)\ \text{and}\ A=\left(\begin{array}{cc}A_{11}&A_{12}\\A_{21}&A_{22}\end{array}\right),
\end{equation}
where $A_{11}\in\mathbb R^{r\times r}$.
Finally, we let
\begin{equation}\label{betadef}
\overline \beta(\varphi,t)=-A_{11}(\varphi,t)-\overline \kappa(\varphi,t)A_{21}(\varphi,t),
\end{equation}
which will appear often below.\\
{\bf Step 1:} Interpret (\ref{grad}) and $\mathcal C_2$ condition (\ref{Urestrict}) on $U_{s,t},A$.\\
Suppose $u\in[s,t]$. 
By (\ref{grad}) as well as the mean value theorem 
\begin{equation}
\left(\!\begin{array}{c}\overline\phi^{x,s}(y,u)-\overline\phi^{x,s}(\hat y,u)\\\widetilde\phi^{x,s}(y,u)-\widetilde\phi^{x,s}(\hat y,u)\end{array}\!\right)=\left(\!\begin{array}{cc} I_r & \overline \kappa(\phi^{x,s}(y^*,u),u)\\
0 & 0 \end{array}\!\right)
U^{-1}_{s,u}\phi^{x,s}\cdot(y-\hat y)
\end{equation}
with $y^*\in\mathcal N^{x,s}$ for $y,\hat y\in\mathcal N^{x,s}$ and any possible representation $\phi^{x,s}$.
Hence, $\overline\phi^{x,s}(y,u)\neq\overline\phi^{x,s}(\hat y,u)$ implies $y\neq\hat y$.
Therefore, it follows from (\ref{Urestrict}) that 
$U_{s,t}\phi^{x,s}$ can not depend upon $\overline\phi^{x,s}( y,u)$
for any $u\in[s,t]$, 
which implies $U_{s,t}\phi\doteq U_{s,t}\widetilde \phi$ only depends on $\widetilde \phi,t$.
This also means by (\ref{SecondA}) that 
\begin{equation}\label{FirstA}
A(\varphi,t)=\frac{d}{dt}  U_{u,t}\widetilde \phi^{\varphi,u}\big|_{u=t}.
\end{equation}
{\bf Step 2:} Interpret commutator conditions on $\overline \kappa,h$.\\
Let $e_i$ denote the $i^{th}$ column of $I_{p}$ so $\sigma_i=e_i$ for $i\le r$.
We have by (\ref{braket2}), that
\begin{equation}
\left(\begin{array}{c}\nabla_\varphi\overline \kappa_{j-r}\\0\end{array}\right)e_i=0\ \ \forall i\in\{1,2,...,r\},j\in{r+1,...,d},
\end{equation}
which establishes that $\overline \kappa(\widetilde \varphi,t)$ can only depend upon $\widetilde \varphi,t$.
This is the only restriction on $\overline \kappa$ from (\ref{braket2}).
By (\ref{h}), we find
\begin{equation}
\nabla_\varphi\left(\begin{array}{c}\overline h\\\widetilde h\end{array}\right)\sigma_j-\nabla_\varphi\sigma_j\left(\begin{array}{c}\overline h\\\widetilde h\end{array}\right) =\left(\begin{array}{cc}\overline \beta&\partial_t\overline \kappa-A_{12}-\overline \kappa A_{22}\\0&0\end{array}\right)_j
\end{equation}
so $\nabla_{\overline \varphi} \widetilde h=0$, implying $\widetilde h(\varphi)\in C^1(\widetilde D_T,\mathbb R^{p-r})$ only depends upon $\widetilde \varphi,t$, and
\begin{equation}\label{SpecCom1}
\nabla_{\overline{\varphi}} \overline h=\overline \beta,
\end{equation}
\begin{equation}\label{SpecCom2}
\overline \beta\, \overline \kappa=[(\nabla_{\widetilde \varphi}\overline \kappa_1)\widetilde h,...,(\nabla_{\widetilde \varphi}\overline \kappa_{d-r})\widetilde h] +\partial_t\overline \kappa-A_{12}-\overline \kappa A_{22}.
\end{equation}
Hence, it follows from (\ref{grad},\ref{dtphi},\ref{initphi}) that $\widetilde\phi^{x,s}$ satisfies
\begin{eqnarray}
\nabla _y\widetilde\phi^{x,s}(y,t)& = & 0, \label{vargrad} \\
\partial _t\widetilde\phi^{x,s}(y,t) & = & \widetilde h(\widetilde\phi^{x,s}(y,t),t), \label{dtivarphi}\\
\widetilde\phi ^{x,s}(0,s) & = & \widetilde x, \label{initvarphi}
\end{eqnarray}
which implies that $\widetilde \phi$ does not depend upon $\overline \phi$ nor $y$.
Moreover, by (\ref{FirstA}) and (\ref{betadef}), we conclude that $A(\varphi,t)\doteq A(\widetilde \varphi,t)$
and $\overline \beta(\varphi,t)\doteq\overline \beta(\widetilde \varphi,t)$ only depend on $\widetilde \varphi,t$.\\
{\bf Step 3:} Determine possible $h,b$.\\
By (\ref{SpecCom1}), we find
\begin{equation}\label{h1}
\overline h(\overline \varphi,\widetilde \varphi,t)= \overline \beta(\widetilde \varphi,t)\overline \varphi+\overline \theta(\widetilde \varphi,t)
\end{equation}
for some $C^1$-function $\overline \theta$.
Hence, the possible $h(\overline \varphi,\widetilde \varphi,t)=\left(\begin{array}{c}\overline h(\overline \varphi,\widetilde \varphi,t)\\\widetilde h(\widetilde \varphi,t)\end{array}\right)$ are:
\begin{equation}\label{possibleh}
\begin{array}{c}\widetilde h\in C^1(\widetilde D_T,\mathbb R^{p-r}),\\ 
\overline h\in \left\{\overline \theta(\widetilde \varphi,t)+
\overline \beta(\widetilde \varphi,t)\overline \varphi:\overline \beta\in C^1(\widetilde D_T,\mathbb R^{r\times r});\overline \theta\in C^1(\widetilde D_T,\mathbb R^r)\!\right\}
\end{array}.
\end{equation}
From (\ref{bh}) and fact $\overline \kappa(\widetilde \varphi,t)$ only depends on $\widetilde \varphi,t$, we find that
\begin{equation}
b=h+{\frac 12}\sum_{j=1}^d\{\nabla_\varphi \sigma _j\}\sigma _j=h.
\end{equation}
{\bf Free Parameters:} $A_{21}$, $A_{22}$, $\overline \kappa$, $\overline \beta$, $\overline \theta$ and $\widetilde h$ can be anything (subject to dimensionality and dependency on only $\widetilde \varphi,t$).
$A_{12}$ is then determined by (\ref{SpecCom2})
and $A_{11}$ by (\ref{betadef}).
$\overline \beta$ and $\overline \theta$ also determine the possible $\overline h$ above and $\phi^{x,s}$ below.
Different choices of $\overline \kappa$, $\overline \beta$, $\overline \theta$ and 
$\widetilde h$ will result in different solutions.
However, there is no loss in generality in taking $A_{21},A_{22}$ to
be zero.\\
{\bf Step 4:} Interpret differential system for $\phi^{x,s}$.\\
Since $\phi^{x,s}=\left(\begin{array}{c}\overline \phi\\\widetilde \phi\end{array}\right)$ satisfies (\ref{dtphi},\ref{initphi}), $\widetilde \phi$ must be of the form
\begin{equation}\label{phi2evolve}
\partial_t\widetilde \phi = \widetilde h(\widetilde \phi,t),\ \ \text{s.t.}\ \widetilde \phi(s)=\widetilde x.
\end{equation}
We let $\widetilde X_t$ denote the solution of this differential equation.
Next, since $\phi$ satisfies (\ref{grad}), $\overline \phi$ must be of the form
\begin{equation}\label{phiphi}
\overline \phi^{x,s}(y,t) = \overline c(t) + \left[ I_r \ \ \overline \kappa( \widetilde X_t,t)
\right]U^{-1}_{s,t}\widetilde \phi^{x,s}y,
\end{equation}
for some $\overline c \in C^1([0,T);\mathbb R^r)$.
Differentiating in $t$, noting by (\ref{FirstA}) (with $\widetilde \varphi=\widetilde X_t$) that
\begin{equation}\label{AAsim}
A(\widetilde X_t,t)=
\frac{d}{dt} U_{u,t}\widetilde \phi^{\widetilde X_u,u}\big|_{u=t},
\end{equation}
and using (\ref{phiphi},\ref{AAsim},\ref{Ugroup},\ref{phi2evolve},\ref{SpecCom2},\ref{betadef}), one has (with $U^{-1}_{s,t}=U^{-1}_{s,t}\widetilde \phi^{x,s}$) that
\begin{eqnarray}\label{partialtphi}
\!\!&\!\!\!&\!\!\!\partial_t \overline \phi(y,t)\\\nonumber
&\!\!\!=&\!\!\!   \overline c' (t) -\left[ I \ \
\overline \kappa( \widetilde X_t,t)\right]A(\widetilde X_t,t) U^{-1}_{s,t} y
\\
&\!\!\!+&\!\!\!\left[ 0 \ \partial_t\overline \kappa( \widetilde X_t,t)+\nabla_{\widetilde \varphi}\overline \kappa_1(\widetilde X_t,t)\, \widetilde h(\widetilde X_t,t),...,\nabla_{\widetilde \varphi}\overline \kappa_{d-r}(\widetilde X_t,t)\, \widetilde h(\widetilde X_t,t)
\right]U^{-1}_{s,t} y\ \ \nonumber\\
&\!\!\!=&\!\!\!\overline c' (t) +\overline \beta(\widetilde X_t,t)[I\ \ \overline \kappa(\widetilde X_t,t)]U^{-1}_{s,t} y\nonumber\\
&\!\!\!=&\!\!\!\overline c' (t) +\overline \beta(\widetilde X_t,t)(\overline \phi(y,t)-\overline c(t))\nonumber.
\end{eqnarray}
On the other hand, by (\ref{dtphi}) and (\ref{h1})
\begin{equation}\label{phipartialt}
\partial_t \overline \phi(y,t)= \overline \theta(\widetilde X_t,t)+\overline \beta(\widetilde X_t,t)\overline \phi(y,t).
\end{equation}
Comparing (\ref{partialtphi}) and (\ref{phipartialt}), one has that
\begin{equation}\label{ceqn}
\overline c'(t)=\overline \theta(\widetilde X_t,t)+\overline \beta(\widetilde X_t,t)\overline c(t)\ \ \text{subject to }\overline c(s)=\overline{x}.
\end{equation}
{\bf Step 5:} Determine $U$ in terms of $\overline \kappa$, $\overline \beta$ and $\theta$.\\
We just need $A$ to satisfy (\ref{betadef},\ref{SpecCom2}) so there is no loss of generality in taking
\begin{equation}\label{Aderived}
\!\!\!\!\!\!\left(\!\!\begin{array}{cc}A_{11}&A_{12}\\A_{21}&A_{22}\end{array}\!\!\right)(\widetilde\varphi,t)
=\left(\!\begin{array}{cc} -\overline \beta & {
[(\nabla_{\widetilde\varphi}\overline \kappa_1)\widetilde h,...,(\nabla_{\widetilde\varphi}\overline \kappa_{d-r})\widetilde h]+\partial_t\overline \kappa
-\overline \beta\,\overline \kappa
} \\0& 0 \end{array}\!\right)(\widetilde\varphi,t).
\end{equation}
By (\ref{AAsim}), (\ref{Ugroup}) and (\ref{Aderived}), we know
\begin{eqnarray}\label{MatUdef}
\!\!\partial_tU_{s,t}\widetilde X&\!\!\!=&\!\!\!(U_{s,t}\widetilde X)A(\widetilde X_t,t)\\
&\!\!\!=&\!\!\!U_{s,t}\widetilde X\!\nonumber
\left(\!\!\begin{array}{cc} -\overline \beta & \{[(\nabla_{\widetilde \phi}\overline \kappa_1)\widetilde h,...,(\nabla_{\widetilde \phi}\overline \kappa_{d-r})\widetilde h]
+\partial_t\overline \kappa-\overline \beta\,\overline \kappa\}\! \\0 & 0 \end{array}\!\!\right)\!(\widetilde X_t,t)
\end{eqnarray}
subject to $U_{s,s}\widetilde X=U_{s,s}\widetilde x=I_d$.
Now, suppose that $T_{u,t}$ is the two parameter semigroup:
\begin{equation}
\frac{d}{dt}T_{u,t}=-T_{u,t}\,\overline \beta(\widetilde X_t,t),\ \ \forall t\ge u\ \ \text{subject to }T_{u,u}=I_r.
\end{equation}
Then, the solution of (\ref{MatUdef}) is
\begin{equation}
\!\!\!U_{s,t}\widetilde X=
\left(\!\begin{array}{cc} T_{s,t} & T_{s,t}\overline \kappa(\widetilde X_t,t)-\overline \kappa(\widetilde X_s,s) \\0 & I_{d-r} \end{array}\!\right),
\end{equation}
and so
\begin{equation}
U_{s,t}^{-1}\widetilde X=
\left(\!\begin{array}{cc} T^{-1}_{s,t} & T^{-1}_{s,t}\overline \kappa(\widetilde X_s,s)-\overline \kappa(\widetilde X_t,t) \\0 & I_{d-r} \end{array}\!\right).
\end{equation}
Moreover, it follows by (\ref{ceqn}) that $\overline c$ can also be expressed in
terms of $T_{s,t}^{-1}$.
{\bf Step 6:} Solution Algorithm.\\
\begin{enumerate}
\item[a:]
Check $\overline \kappa$ only depends upon $\widetilde \varphi,t$.
This must be true by Step 2.
\item[b:]
Choose any functions $\overline \beta\in C^1(\widetilde D_T,\mathbb R^{r\times r});\overline \theta\in C^1(\widetilde D_T,\mathbb R^r)$ and $\widetilde h\in C^1(\widetilde D_T,\mathbb R^{p-r})$ for drift of the form
$b(\overline \varphi,\widetilde \varphi,t)=h(\overline \varphi,\widetilde \varphi,t)=
\left(\begin{array}{c}\overline \theta(\widetilde \varphi,t)+\overline \beta(\widetilde \varphi,t)\overline \varphi\\\widetilde h(\widetilde \varphi,t)\end{array}\right)$.
These are the only possible drifts by Step 3.
\item[c:]
Solve
\begin{eqnarray*}
\widetilde X'_t&\!\! =&\!\!\widetilde h(\widetilde X_t,t)\ \ \text{subject to }\widetilde X_s = \widetilde x
\end{eqnarray*}
\item[d:]
Solve
\begin{equation}
\frac{d}{dt}T_{s,t}=-T_{s,t}\, \overline\beta(\widetilde X_t,t),\ \ \forall t\ge s\ \ \text{subject to }T_{s,s}=I_r.
\end{equation}
Then, set
\begin{eqnarray}
\!\!U_{s,t}\widetilde X&\!\!=&
\!\!\left(\!\begin{array}{cc} T_{s,t} & T_{s,t}\overline \kappa(\widetilde X_t,t)-\overline \kappa(\widetilde X_s,s) \\0 & I_{d-r} \end{array}\!\right),\\
\!\!U_{s,t}^{-1}\widetilde X&\!\!=&
\!\!\left(\!\begin{array}{cc} T^{-1}_{s,t} & T^{-1}_{s,t}\overline \kappa(\widetilde X_s,s)-\overline \kappa(\widetilde X_t,t) \\0 & I_{d-r} \end{array}\!\right),\\
\overline c(t) &\!\!=& T^{-1}_{s,t}\overline{x}
+ T^{-1}_{s,t} \int_s^tT_{s,u}\overline \theta(\widetilde X_u,u)du.	
\end{eqnarray}
\item[e:]
Divide $\phi=\left(
            \begin{array}{c} \overline \phi \\ \widetilde \phi \end{array}
           \right)$ and set
$ \widetilde \phi(t) =\widetilde X_t $,
\begin{eqnarray*}
	\overline \phi(y,t)&=&\overline c(t) + \left[ I_r \ \ \overline \kappa(\widetilde X_t,t)
\right](U^{-1}_{s,t}\widetilde X)y.
\end{eqnarray*}
\end{enumerate}
\end{ex}
The preceding example was intuitively pleasing:
We showed you could indeed represent \emph{linear} SDEs using a single Gaussian stochastic integral.
Further, we showed that we could \emph{append} an ordinary differential equation
($d\widetilde X_t = \widetilde h(\widetilde X_t)dt$) and use its solution
within the coefficients of the stochastic differential equation.
Finally, we showed how to construct the solution.
While none of this is surprising, it does explain our necessary and sufficient conditions.
In the next section, we will show how to combine this example with diffeomorphisms to handle the general case
with nonlinear coefficients.


\section{Construction Results and Examples} \label{appli}
When one explicit solution exists, there will be a whole class of
such solutions corresponding to distinct $b$'s. We now identify the
$b$'s, $\phi$'s and $U$'s for these solutions corresponding to a
given $\sigma$.
This is done by using local diffeomorphisms to convert the
general case to the case of Example \ref{CoreEG}.
The idea is based upon the following simple lemma.

\begin{lem}\label{dc}
Suppose $D \subset \mathbb R^p$ is a domain, $T>0$, $D_T=D\times[0,T)$,
$\widehat\Lambda\doteq\left(\begin{array}{c}\Lambda_t\\t\end{array}\right):
D_T\rightarrow\widehat\Lambda(D_T)\subset\mathbb R^{p+1}$ is a $C^2$-diffeomorphism 
and $\sigma, b, h, \{\phi^{x,s}\}_{(x,s)\in D_T}$, 
$\{U_{s,t}\phi^{x,s}\}_{(x,s)\in D_T,s\le t<T}$, $A$ satisfy
Conditions $\mathcal C_1, \mathcal C_2, \mathcal C_3$ as well as equations (\ref{bh},\ref{SecondA}).
Let $\widehat D_T=\widehat\Lambda(D_T)$,
\begin{eqnarray*}
\widehat{\sigma} &=& \{(\nabla_\varphi\Lambda_t)\sigma\}\circ\widehat\Lambda^{-1},\ \ \widehat{h} = \{(\nabla_\varphi\Lambda_t)h\}\circ\widehat\Lambda^{-1},\\
\widehat{b} &=& \left\{(\nabla_\varphi\Lambda_t)b+\frac12\sum_{j=1}^d\sum_{i,k=1}^p
(\partial_{\varphi_i}\partial_{\varphi_k}\Lambda_t)\sigma_{i,j}\sigma_{k,j}\right\}\circ\widehat\Lambda^{-1},\ \
\\
\widehat{\phi}^{x,s}(y,t) &=& \Lambda_t\circ\phi^{\widehat\Lambda^{-1}(x,s)}(y,t),\\
\widehat{U}_{s,t}\widehat{\phi}^{x,s} &=& U_{s,t}\phi^{\widehat\Lambda^{-1}(x,s)},\\
\widehat{A}&=&A\circ\widehat\Lambda^{-1}.
\end{eqnarray*}
Then, $\widehat \sigma, \widehat b, \widehat h, \{\widehat \phi^{x,s}\}_{(x,s)\in \widehat D_T}, \widehat U, \widehat A$ satisfy
Conditions $\mathcal C_1, \mathcal C_2, \mathcal C_3$ as well as equations (\ref{bh},\ref{SecondA}) on $\widehat D_T$.
Moreover,
\begin{itemize}
\item[i)] $\widehat \phi,\widehat U$ is a $(x,s,\widehat \sigma,\widehat h)$-representation for each $(x,s)\in\widehat D_T$ if and only if $\phi,U$ is a $(x,s,\sigma,h)$-representation for each $(x,s)\in D_T$.
\item[ii)] (\ref{braket2}) holds if and only
\begin{equation}\label{braket2hat}
(\nabla_\varphi \widehat \sigma _k)\widehat \sigma _j=(\nabla_\varphi \widehat \sigma _j)\widehat \sigma _k,
\mbox{ on } \widehat D_T\mbox{ for all } j,k\in \{1,...,d\}.
\end{equation}
\item[iii)] (\ref{h}) holds if and only
\begin{equation} \label{hhat}
(\nabla_\varphi \widehat h)\widehat \sigma _j
=(\nabla_\varphi \widehat \sigma _j)\widehat h+\partial_t\widehat \sigma_j-\widehat \sigma\widehat  A_j,\mbox{ on } \widehat D_T\mbox{ for all } 1\le j\le d.
\end{equation}
\end{itemize}
\end{lem}
\begin{RM}
In the time-homogeneous case, we can deal with $B$ instead of $A$ and set	
$\widehat{B}=B\circ\Lambda^{-1}_0$.
\end{RM}
\begin{proof}
This lemma follows by direct calculation.
Perhaps, the fastest way to verify the commutator conditions is to think of (\ref{Strato}) as a time-homogeneous
equation
\[
d\left[\begin{array}{c}X_t\\t\end{array}\right]=\left[\begin{array}{c}h(X_t,t)\\1\end{array}\right]dt+\left[\begin{array}{c}\sigma (X_t,t)\\0\end{array}\right]\bullet dW_t,\qquad  \left[\begin{array}{c}X_s\\s\end{array}\right]=\left[\begin{array}{c}x\\s\end{array}\right]
\]
on $[s,\tau^{x,s})$, by appending the trivial equation $t=t$ and thinking of $t$
as an additional state variable.
Then, verifying (\ref{h}) is equivalent to (\ref{hhat}) is the same as
verfying
\begin{eqnarray*} 
&\!\!&\!\!\left(\nabla\left[\begin{array}{c}h\\1\end{array}\right]\right)\left[\begin{array}{c}\sigma_j \\0\end{array}\right]
=\left(\nabla\left[\begin{array}{c}\sigma_j \\0\end{array}\right]\right)\left[\begin{array}{c}h\\1\end{array}\right]-\left[\begin{array}{c}\sigma \\0\end{array}\right]A_j\\ 
	&\!\!\leftrightarrow&\!\!
\left(\nabla\left[\begin{array}{c}\widehat h\\1\end{array}\right]\right)\left[\begin{array}{c}\widehat \sigma_j \\0\end{array}\right]
=\left(\nabla\left[\begin{array}{c}\widehat \sigma_j \\0\end{array}\right]\right)\left[\begin{array}{c}\widehat h\\1\end{array}\right]-\left[\begin{array}{c}\widehat \sigma \\0\end{array}\right]\widehat  A_j,
\end{eqnarray*}
which avoids $\partial_t\sigma_j$ and $\Lambda_t$ if we express $(\widehat h^T,1)^T$ and
$(\widehat \sigma_j^T,0)^T$ in terms of $\widehat\Lambda$.
\end{proof}

The idea behind this lemma is that with some diffeomorphism $\widehat{\sigma}=
\left(\begin{array}{cc} I_r & \overline \kappa\\ 0 & 0 \end{array}\right)$
so we can use Example \ref{CoreEG} to solve for the possible $\hat h$
and the representations $\widehat{\phi}^{x,s}$, $\widehat{U}^{x,s}$.
Unfortunately, it is sometimes impossible to have a single diffeomorphism
for all of $D_T$ and, even when it is possible, we may not know that until
after local diffeomorphisms are constructed and one of them is extendable to all
of $D_T$.

\begin{defn}
Suppose $(x,s)\in D_T$. Then, an $({x},s)$-local
diffeomorphism $(O^{x,s},\widehat{\Lambda}^{x,s})$ is a bijection
$\widehat{\Lambda}^{x,s}:O^{x,s}\rightarrow
\widehat{\Lambda}^{x,s}(O^{x,s})$ such that $\widehat{\Lambda}^{x,s}\in
C^2(O^{x,s};\mathbb R^{p+1})$, where $O^{x,s}\subset D_T$ is a
(relatively open) neighbourhood of $x,s$. 
We define $\n \widehat{\Lambda}^{-1}(\widehat\Lambda(\varphi,t))$ to be
$\left[\nabla\widehat\Lambda (\varphi,t)\right]^{-1}$ for $(\varphi,t)\in O^{
x,s}$.
\end{defn}
We imposed sufficient differentiability on our local diffeomorphisms for
our uses to follow.
Our $(x,s)$-local diffeomorphisms will take the form 
$\widehat\Lambda=\left(\begin{array}{c}\Lambda_t\\t\end{array}\right)$ 
with $\Lambda_t$ being constructed from $\sigma$ under the conditions:
\begin{itemize}
\item[$D$: ]
Let $D\subset\mathbb R^p$ be a bounded convex domain, $T>0$ and 
$D_T=D\times[0,T)$.
\item[$\partial_1$: ]
$\sigma\in C^{ 1}(D_T;{\mathbb R}^{p\times d})$.
\item[$H_r$: ]
The rank of $\sigma$ is $r$ on $D_T$ with the first $r$ rows having full
row rank.
\item[$B$: ] 
$(\nabla_\varphi \sigma_j)\sigma_k -(\nabla_\varphi \sigma_k)\sigma_j =0 $ on $D_T$,
for $1\le j,k \le d$ and $(x,s)\in D_T$.
\end{itemize}
To ensure the row rank part of $H_r$, we can just permute the rows of (\ref{VSDE1}),
amounting to relabeling the $\{X^i_t\}_{i=1}^p$.
\begin{prop}\label{diff1}
Suppose [$D$, $\partial_1$, $H_r$, $B$] hold.
Then,
there exists an $(x,s)$-local diffeomorphism $(O^{x,s},\widehat\Lambda^{x,s})$
and a constant permutation matrix $\pi$ such that
$$
\widehat \sigma \doteq \{(\nabla_\varphi \Lambda_t)\sigma \pi\}\circ \widehat\Lambda ^{-1}=
\left(\begin{array}{cc} I_r & \overline \kappa\\ 0 & 0 \end{array}\right) \in \mathbb R^{p\times d}  \mbox{ on } \widehat\Lambda(O^{x,s}),
$$
where $\overline \kappa \in C^1(\widehat\Lambda(O^{x,s});{\mathbb R}^{r\times (d-r)})$
does not depend on $\varphi_1, \ldots , \varphi_r$.
\end{prop}
\begin{proof}
Provided in the Appendix.
\end{proof}
\begin{RM}
The permutation matrix $\pi$ permutes the columns of $\sigma$.
We label the permuted diffusion coefficient $\sigma^\pi=\sigma\pi$ and note that
\[
	dX_t=b(X_t)dt+\sigma(X_t)dW_t=b(X_t)dt+\sigma^\pi(X_t)dW^\pi_t,
\]
where $W^\pi=\pi^{-1}W$ is a permutation of the Brownian motions $W$.
Also, the Stratonovich drift $h$ remains the same by (\ref{bh}).
\end{RM}
\begin{RM}\label{diffform}
It follows from the proof in the Appendix that the diffeomorphism can have the form 
$\widehat\Lambda=\widehat\Lambda_r\circ\cdots\circ\widehat\Lambda_2\circ\widehat\Lambda_1$ for any diffeomorphisms
$\widehat\Lambda_i:\widehat\Lambda_{i-1}\circ\cdots\circ\widehat\Lambda_2\circ\widehat\Lambda_1(D_T)\rightarrow \mathbb R^{p+1}$ satisfying 
$\{\nabla\widehat\Lambda_{i}\cdots \nabla\widehat\Lambda_2\nabla \widehat\Lambda_1\sigma^\pi_i\}\circ\widehat\Lambda_1^{-1}\circ\widehat\Lambda_2^{-1}\circ\cdots\circ\widehat\Lambda_i^{-1}=e_i$,
where $(e_1\, e_2\,\ldots\,e_p\,e_{p+1})=I_{p+1}$ is the identity matrix. 
However, as will be seen below in Remark \ref{Remark7}, this does not
uniquely define the diffeomorphism.
\end{RM}

Proposition \ref{diff1} immediately provides us our second main theorem.


\begin{thm}\label{r2}
Suppose [$D$, $\partial_1$, $H_r$, $B$] hold, $h\in C^1(D_T;{\mathbb R}^p)$, $(x,s)\in D_T$ and 
$W$ is an $\mathbb R^d$-valued standard Brownian motion.
Then, there exists a stopping time $\tau>s$, a permutation matrix $\pi$ and an $(x,s)$-local diffeomorphism $(O^{x,s},\widehat \Lambda^{x,s})$,
as in Proposition \ref{diff1} and Remark \ref{diffform}, such that
$$
\mbox{i)  }\ \widehat \sigma \doteq \{(\nabla_\varphi \Lambda_t)\sigma^\pi\}\circ \widehat\Lambda ^{-1}=
\left(\begin{array}{cc} I_r & \overline \kappa\\ 0 & 0 \end{array}\right) \in \mathbb R^{p	\times d}  \mbox{ on } \widehat\Lambda(O^{x,s}),
$$
with $\overline \kappa \in C^1(\Lambda(O^{x,s});{\mathbb R}^{r\times (d-r)})$
not depending on $\varphi_1, \ldots , \varphi_r$ and ii) the Stratonovich SDE
\(
dX_t =h(X_t)dt +\sigma(X_t)\bullet dW_{t},\ X_s=x
\)
has a solution  
$X_t=\Lambda^{-1}_t\left(\begin{array}{c}\overline X_t\\\widetilde X_t\end{array}\right)$ on $[0,\tau]$ 
if and only if the simpler SDE 
\[
\!\!d\left[\begin{array}{c}\overline X_t\\\widetilde X_t\end{array}\right]
= \widehat h\left(\!\begin{array}{c}\overline X_t\\\widetilde X_t\end{array}\!\right)dt
+\left(\begin{array}{cc} I_r & \overline \kappa\\ 0 & 0 \end{array}\right) dW^\pi_t,\ \left[\begin{array}{c}\overline X_s\\\widetilde X_s\end{array}\right]=\Lambda_s(x)
\]
has a solution on $[0,\tau]$, where
$\widehat h=(\nabla_\varphi \Lambda_t h+\partial_t\Lambda_t)\circ\widehat\Lambda^{-1}$.
\end{thm}

We stated the simpler SDE in terms of It\^{o} integration.
However, it follows by (\ref{bh}) and the nature of $\overline\kappa$ that
this equation would have exactly the same form in terms of Stratonovich integration.

In this theorem we do not have a commutator condition for $h$ so we can not
guarantee the simple form of $\widehat h$ as in Example \ref{CoreEG}.
This means that $\widetilde X$ is not in general deterministic nor is $\overline X$
necessarily Gaussian.
We also impose slightly stronger conditions on $\sigma$ compared to
Theorem \ref{r1} but gain
information about the representation as local diffeomorphisms.

\cite{Kouritzin16} solves for a local diffeomorphism $\widehat \Lambda$ of
the form stated in Remark \ref{diffform} corresponding
to the (extended) Heston model, shows that it exists globally, 
finds the corresponding $\widehat h$, and solves the SDEs.
The use of the extended model means that our explicit Heston SDE solutions are weak not strong
because the real Heston model corresponds to just part of the extended model
that includes \emph{extra randomness}.
Also, this approach only works when a condition is imposed on the Heston parameters.
When this condition is not true, one can still obtain an explicit weak solution
by using Likelihoods and Girsanov's theorem to convert to the case where the
condition is true.

For our final main result, we add back the commutator
condition for $h$, and characterize all the solutions $X^{x,s}_t=\phi^{x,s}(Y_t,t)$ to
(\ref{SDE4}) via Example \ref{CoreEG}.
We do this through our basic set of parameters for $(x,s)$:

\begin{defn}\label{def3}
Let ${\mathcal P} = {\mathcal P}^{x,s}_{\sigma}$ be the set of all
$(\widehat \Lambda, \overline \kappa, \overline \beta, \overline \theta, \widetilde h,\pi)$ such that
\begin{enumerate}
\item[P0)] $\pi$ is a constant permutation matrix.

\item[P1)] $(O^{x,s},\widehat \Lambda^{x,s})$ is a
$(x,s)$-local diffeomorphism, where $\widehat \Lambda(\varphi,t) =
\mat{\Lambda_t(\varphi)}{t}$.
For convenience, we let $\Lambda_t = \mat{\overline\Lambda_t}{\widetilde\Lambda_t}$ with
$\overline\Lambda_t \in \mathbb R^r$;

\item[P2)] $\overline \kappa \in C^1(\widehat \Lambda(O);{\mathbb R}^{r\times(d- r)}) $ depends only on
$\varphi_{r+1}, \ldots, \varphi_p$, and $t$;

\item [P3)] $\{(\n_\varphi \Lambda_t )\sigma^\pi\}\circ (\widehat\Lambda)^{-1} =
\left(\begin{array}{rl} I_r & \overline \kappa \\
0 & 0\end{array}\right)$ on $\widehat\Lambda(O)$;

\item[P4)] $\overline \beta \in C^1(\widehat \Lambda(O);{\mathbb R}^{r\times r}) $ depends only
on $\varphi_{r+1}, \ldots, \varphi_p$, and $t$;

\item[P5)] $\overline \theta \in
C^1(\widehat \Lambda(O);{\mathbb R}^r)$ depends only on
$\varphi_{r+1}, \ldots, \varphi_p, t$;

\item[P6)]
$\widetilde h\in C^1(\widehat \Lambda(O);{\mathbb R}^{p- r})$ depends only on
$\varphi_{r+1}, \ldots, \varphi_p, t$.
\end{enumerate}
\end{defn}

To each $(\widehat \Lambda, \overline \kappa, \overline \beta, \overline \theta, \widetilde h,\pi) \in {\mathcal P}$,
we
associate the following functions:
\begin{equation}\label{frel}
\begin{cases}
\widetilde X=\widetilde X^{x,s} \in {\mathbb R}^{p-r} \mbox{ uniquely solves } \frac{d}{dt} \widetilde X_t =
\widetilde h (\widetilde X_t,t),\;
\widetilde X_s =\widetilde \Lambda_s(x);\\
G(t) = \left(I_r \Big{|}\; \overline \kappa( \tilde X_t,t)\right)\in {\mathbb R}^{r\times d};\\
\frac{d}{du}T_{s,u}=-T_{s,u}\,\overline \beta(\widetilde X_u,u),\ \ \forall u\ge s\ \ \text{subject to }T_{s,s}=I_r;\\
\!U_{s,u}\widetilde X=
\left(\!\begin{array}{cc} T_{s,u} & T_{s,u}\overline \kappa(\widetilde X_u,u)-\overline \kappa(\widetilde X_s,s) \\0 & I_{d-r} \end{array}\!\right);\\
\!U^{-1}_{s,u}\widetilde X=
\left(\!\begin{array}{cc} T^{-1}_{s,u} & T^{-1}_{s,u}\overline \kappa(\widetilde X_s,s)-\overline \kappa(\widetilde X_u,u) \\0 & I_{d-r} \end{array}\!\right);\\
\overline c_s(t) = T^{-1}_{s,t}\overline{\Lambda}_s(x)
+ T^{-1}_{s,t} \int_s^tT_{s,u}\overline \theta(\widetilde X_u,u)du.
\end{cases}
\end{equation}


\medskip
The following theorem follows from Theorem \ref{r2}, Theorem \ref{r1} (so the explicit solution
implies $B$ above) and Example \ref{CoreEG}.
In particular, we must have
\begin{equation}\label{possibleh2}
(\nabla_\varphi \Lambda_t h+\partial_t\Lambda_t)\circ\widehat\Lambda^{-1}
=\left(\begin{array}{c}\overline h(\overline \varphi,\widetilde \varphi,t)\\\widetilde h(\widetilde \varphi,t)\end{array}\right)
=
\left(\begin{array}{c}
\overline \theta(\widetilde \varphi,t)+
\overline \beta(\widetilde \varphi,t)\overline \varphi\\\widetilde h(\widetilde \varphi,t) 
\end{array}\right),
\end{equation}
which gives our possible drifts $h$ in the following theorem.
\medskip
\begin{thm}\label{r3}
Suppose [$D$, $\partial_1$, $H_r$] hold, $(x,s)\in D_T$
and $X_t^{x,s} = \phi^{x,s}\left(\int_s^t U_{s,u}\phi^{x,s}dW^\pi_u,t
\right)$, with $\phi,U$ satisfying $\mathcal C_1$, $\mathcal C_2$, $\mathcal C_3$, solves (\ref{SDE4}) up to
some stopping time
$\tau^{x,s}>s$.
Then, there exists $((O^{x,s},\widehat \Lambda^{x,s}), \overline \kappa, \overline \beta, \overline \theta, \widetilde h,\pi)
\in {\mathcal
P}^{x,s}_{\sigma}$, and related functions $\widetilde X, G, U, \overline c$ defined
by (\ref{frel}), such that
\begin{equation} \label{hdef}
h= [\n_\varphi \Lambda_t]^{-1}\left\{
\left[\begin{array}{c}\overline\theta (\widetilde X_t,t)\\
\widetilde h (\widetilde X_t,t)\end{array}\right] -\partial_t \Lambda_t
 + \mat{\overline\beta(\widetilde X_t,t)\ \overline \Lambda_t}{0} \right\} \mbox{ on } O^{x},
\end{equation}
\begin{equation} \label{fdef}
\phi^{x,s}(y,t) = \phi_{(\widehat\Lambda, \overline \kappa, \overline\beta,\overline \theta, \widetilde h)}(y,t) =
\Lambda_t^{-1}\left(  \mat{\overline c_s(t) +G(t)(U^{-1}_{s,t}\widetilde X)y}{\widetilde X_t}\right)
\end{equation}
on ${\mathcal N}^{x} = \left\{(y,t):
\mat{\overline c_s(t) +G(t)U^{-1}_{s,t}\widetilde Xy}{\widetilde X_t} \in \Lambda_t(O^{x,s})\right\}$.
Finally, if $\breve\pi$, $\breve\Lambda$ and $\breve\kappa$ also satisfies P0--P3,
then there exist $\breve\beta, \breve\theta,\breve h$ such that
$ (\breve\Lambda,\breve \kappa,\breve\beta, \breve\theta,\breve h, \breve\pi) \in {\mathcal P}$,
$b_{(\breve\Lambda, \breve\kappa,\breve\beta,\breve \theta,\breve h, \breve\pi)}
= b_{(\widehat\Lambda, \overline\kappa, \overline\beta, \overline\theta,\widetilde h)}$, and
$\phi_{(\breve\Lambda, \breve\kappa,\breve\beta, \breve\theta,\breve h, \breve\pi)}
= \phi_{(\widehat\Lambda, \overline\kappa, \overline\beta, \overline\theta,\widetilde h)}$.
\end{thm}

\begin{RM}\label{DiffFind}
For the sake of brevity in the examples below, we will just give local diffeomorphisms
satisfying P3) above.
However, as is shown in our companion paper \citet{Kouritzin16}, it is
often possible to solve for them using the technique used in the proof
of Proposition \ref{diff1} herein.
\end{RM}

\begin{RM}\label{Remark7}
To illustrate the need of the final statement of Theorem \ref{r3}, we take for example,
$\sigma (x) = x \in {\mathbb R}^p$. Then, any  $L \in C^1({\mathbb R}^p)$ depending on $x_2/x_1,
\ldots, x_p/x_1$ satisfies
$(\nabla L)\sigma =0$.  Therefore, $\widehat\Lambda $ and hence the parameter set
is not unique but we can create the same $b,\phi$ from any consistent
$\overline \kappa ,\widehat\Lambda .$
\end{RM}

\subsection{One Dimensional Case} 
Suppose $d=p=r=1$, $D\subset\mathbb R$ and $x\in D$.
Then, $\overline\kappa,\widetilde h$ do not exist and
$\overline\beta,\overline\theta$ only depend on $t$.
Moreover, $U_{s,t}=T_{s,t}=e^{-\int_s^t \overline\beta(u)du}$,
$\overline c_s(t)=T_{s,t}^{-1}\overline{\Lambda}_s(x)+T_{s,t}^{-1}\int_s^t T_{s,u}\overline\theta(u)du$
and the diffeomorphism can be taken as
$\Lambda_t(\varphi) = \int\frac{1}{\sigma(\varphi ,t)}d\varphi $.
One then finds by (\ref{bh},\ref{frel},\ref{hdef},\ref{fdef}) that
the corresponding diffusion drift $b$ and explicit solutions are 
\begin{equation}\label{scalardrift}
b(\varphi,t) = \sigma(\varphi,t)\left\{\overline\theta(t) +\overline\beta(t)\Lambda_t(\varphi) - \partial_t \Lambda_t\right\}+\frac{1}{2}\sigma(\varphi,t)\partial_\varphi \sigma(\varphi,t)
\end{equation}
\begin{equation}\label{scalarsoln}
X_t = \Lambda_t^{-1}\left[\left\{\Lambda_s(x)+ \int_s^t
T_{s,u} \overline\theta(u)du +  \int_s^t T_{s,u} dW_u\right\}\Big{/}T_{s,t}\right].
\end{equation}

\begin{ex}[Time-varying Cox-Ingersoll-Ross model]
Suppose $\overline\theta, \overline\beta$ and continuously differentiable
$s(t)>0$ are chosen and
$\sigma(\varphi,t) = s(t) \sqrt{\varphi}$.
Then, 
\(
\Lambda_t(\varphi) = \frac{2\sqrt{\varphi}}{s(t)} 
\),
\(
\Lambda_t^{-1}(z)=\left(\frac{zs(t)}2\right)^2
\)
and the possible It\^{o} drifts are
\[
b(\varphi,t)=\overline\theta(t)s(t)\sqrt{\varphi}+2\left(\overline\beta(t)+\frac{\dot{s}(t)}{s(t)}\right)\varphi+\frac{s^2(t)}4.
\]
The explicit solutions are then
\begin{eqnarray}\label{CIRsoln}
\!\!X_t^{x,s} &\!\!= &\!\!\bigg|\frac{s(t)}{s(s)}e^{\int_s^t \overline\beta(v)dv}\sqrt{x}\\
	\nonumber
&\!\!+&\!\!\frac{s(t)}2\left\{ \int_s^t
	e^{\int_u^t \overline\beta(v)dv} \overline\theta(u)du +  \int_s^t e^{\int_u^t \overline\beta(v)dv} dW_u\right\}\,\bigg|^2.
\end{eqnarray}
In the case $s(t)=\sigma,\overline \theta$ and $\overline\beta$ are taken constant, 
we get
$$
\!X_t^{x,s} = \frac{1}{4}\left\{
2e^{\overline \beta (t-s)}\sqrt{x} 
+ \frac{\overline \theta \sigma}{\overline \beta} (e^{\overline \beta (t-s)}-1) + 
\sigma\! \int_s^t\! e^{\overline \beta(t-u)}dW_u
\right\}^2
$$
solves
$$
dX_t^{x,s} = \left(\sigma^2/4 +2\overline \beta X_t^{x,s} +\sigma \overline\theta \sqrt{X_t^{x,s}}
\right)dt+  \sigma \sqrt{X_t^{x,s}} dW_t, \ X_s=x
$$
as long as $X_t^{x,s}>0$.
This solves the usual CIR model
\begin{equation}\label{eq:feller}
dX_t = \alpha \left(\beta  -X_t \right)dt+  \sigma \sqrt{X_t} dW_t.
\end{equation}
when $\overline\theta =0$, $\alpha = 2\overline \beta$, $\beta = \sigma^2/(8\overline \beta)$.
Now, set $Y_t  = \sqrt{X_t}$, where $X$ solves \eqref{eq:feller}
with $\sigma^2 = 4\alpha\beta$, and $\tau =
\inf\{t >0; X_t = 0\}$. 
It is well known that $P(\tau<\infty)=1$.
Then, 
\begin{eqnarray}
dY_t  &=& \frac{1}{8Y_t}\left(4\alpha\beta - \sigma^2\right)dt -
\frac{\alpha}{2} Y_t dt+
\frac{\sigma}{2} dW_t \nonumber\\
&=&  - \frac{\alpha}{2} Y_t dt+ \frac{\sigma}{2}
dW_t,\label{eq:OU}
\end{eqnarray}
by It\^o's formula.
However, since \eqref{eq:OU}
defines a Gaussian process and $Y$ must be non-negative, one cannot have $Y_t$ defined by
\eqref{eq:OU} unless $t < \tau$.
This explains why we first look for explicit \emph{local} solutions.
\end{ex}

\subsection{Square Non-Singular Case} 
\label{square}
Suppose that $d=p=r$, $\sigma = \sigma(\varphi,t)$ is a $d\times d$ non-singular
continuously-differentiable matrix satisfying (\ref{braket2}), $D\subset\mathbb R^p$ 
and $x\in D$. 
Again, we apply Theorem \ref{r3} and find $\overline\kappa,\widetilde h$ do not exist while
$\overline\beta,\overline\theta$ only depend on $t$.
Also, there is a local diffeomorphism
$\widehat \Lambda=\left( \begin{array}{c}\Lambda_t\\t\end{array}\right)$ such that $ \n_\varphi \Lambda_t(\varphi) = [\sigma(\varphi,t)]^{-1}$, and
all explicit solutions are of the form $\phi^{x,s} (t,y) =
\Lambda_t^{-1}\left(\overline c_s(t) + U^{-1}_{s,t}y\right)$, where
\[
U_{s,t} =
-\int_s^t U_{s,u} \overline\beta(u)du +I \text{ and } \overline c_s(t)
=U^{-1}_{s,t}\left\{\Lambda_s(x)+ \int_s^t U_{s,u}\overline\theta(u)du \right\}
\]
for some $\overline\theta \in C([0,T);\mathbb R^d)$ and $\overline\beta\in
C^1([0,T),\mathbb R^{d\times d})$. 
The resulting drift is 
$$
b(\varphi,t) = \sigma(\varphi,t)\left\{\overline\theta(t) + \overline\beta(t)\Lambda_t(\varphi) - \partial_t \Lambda_t(\varphi)\right\}
+\frac{1}{2}\sum_{j=1}^d (\n_\varphi \sigma_j(\varphi,t))\sigma_j(\varphi,t).
$$

\begin{ex}\emph{Geometric Brownian motions:}
Take $\sigma_{ij}(\varphi) = \varphi_i
\gamma_{ij}$ with $\gamma$ non-singular and $D = (0,\infty)^d$. 
Then, $\sigma$ satisfies the
commutation condition (\ref{braket2}) since $ [ (\n_\varphi \sigma_j)\sigma_k]_i = \varphi_i
\gamma_{ij}\gamma_{ik}$, and the diffeomorphism can be chosen as
$\Lambda(\varphi) = \Lambda_t(\varphi) =\gamma^{-1} \mat{\log{\varphi_1} \\ \vdots}{\log{\varphi_d}}$.
$\Lambda$'s image is $\mathbb R^d$, so
$\Lambda^{-1}(z) = \mat{e^{(\gamma z)_1}\\ \vdots}{e^{(\gamma z)_d}}$ is defined everywhere
and $\phi_i^{x,s}(y,t) = \exp\left[ \gamma \{\overline c_s(t) + U^{-1}_{s,t}y\}\right]_i$.
The possible drifts satisfy
$$
b_i(\varphi,t) = \varphi_i \left\{\alpha_i(t) -\sum_{j=1}^d B_{ij}(t)
\log{\varphi_j}\right\},
$$
for $1\le i \le d$,
where $B(t) =  \gamma\overline\beta(t) \gamma^{-1}$, 
and $\alpha_i(t)=\frac12[\gamma
\gamma^\top]_{ii}+ [\gamma\overline\theta(t)]_i$.
\end{ex}

\begin{ex}
\emph{Diffeomorphism example:}
In the previous examples, we started with $\sigma$.
Suppose instead we had a diffeomorphism
$$
\Lambda (\varphi_1,\varphi_2) = \Lambda_t(\varphi_1,\varphi_2) = \mat{\frac{\pi}{2}  + \arcsin( \log{\varphi_1
\varphi_2}-1)}{\frac{\pi}{2}  + \arcsin( \frac{2\varphi_2}{\varphi_1}-1)}
$$
on $ 1 < \varphi_1 \varphi_2 < e$, $ \varphi_2 \le \varphi_1$.  
Then, the possible full rank $\sigma$'s satisfy $\sigma=(\nabla_\varphi \Lambda)^{-1}$ i.e.\
\begin{equation}
\sigma(\varphi_1,\varphi_2) = \left(\!\begin{array}{cc}
\frac{\varphi_1}{2}\sqrt{2\log{\varphi_1 \varphi_2}-(\log{\varphi_1 \varphi_2})^2}
& -\frac{\varphi_1}{2\varphi_2}\sqrt{\varphi_2(\varphi_1-\varphi_2)}\\
& \\
\frac{\varphi_2}{2}\sqrt{2\log{\varphi_1 \varphi_2}-(\log{\varphi_1 \varphi_2})^2} &
-\frac{1}{2}\sqrt{\varphi_2(\varphi_1-\varphi_2)}
\end{array}\!\right)
\end{equation}
so $(\n \Lambda)\sigma  = I_2$ and $\sigma$ satisfies (\ref{braket2}) by Lemma \ref{dc} ii).
The possible Stratonovich (time-dependent) drifts $h(\varphi_1, \varphi_2,t)$ are
\begin{equation}
\sigma(\varphi_1, \varphi_2)\!\left(\!\!\begin{array}{c}
\overline\theta_1(t)+\overline\beta_{11}(t)(\frac{\pi}{2} \! +\! \arcsin( \log{\varphi_1
\varphi_2}-\!1) )
-\!\overline\beta_{12}(t)(\frac{\pi}{2} \! + \!\arcsin( \frac{2\varphi_2}{\varphi_1}-\!1)\!)\\
\overline\theta_2(t)+\overline\beta_{21}(t)(\frac{\pi}{2} \! +\! \arcsin( \log{\varphi_1
\varphi_2}-\!1) )
-\!\overline\beta_{22}(t)(\frac{\pi}{2} \! +\! \arcsin( \frac{2\varphi_2}{\varphi_1}-\!1)\!)
\end{array}\!\!\!\right)	
\end{equation}
while $U_{s,t},\overline c_s$ satisfy the equations at the start of Subsection
\ref{square}.
\end{ex}

\subsection{Non-Square Case} 
Our most important example is probably the \emph{Extended Heston}
model of our companion paper \cite{Kouritzin16}.
It is non-square.
However, we provide a second interesting non-square example herein.
\begin{ex}[Heisenberg group] 
Let $\overline x\in\mathbb R^d$ and  $\widetilde x\in\mathbb R$ be the components
of the starting point, $A=A(t)$ be a $\mathbb R^{d\times d}$ continuously differentiable
matrix function and 
$\sigma(\varphi,t) = \sigma(\xi,z,t) =
\mat{I_d}{(A(t)\xi)^\top}$, where $\xi\in \mathbb R^d$, $z\in \mathbb R$. 
Then, $\sigma$ has rank $r=d$. 
The solution to $dX_t = \sigma(X_t,t)dW_t$ is known as the Brownian motion
on the Heisenberg group.
Moreover,
$$
(\n_\varphi\sigma_j)\sigma_k- (\n_\varphi\sigma_k)\sigma_j =
\mat{0}{A_{jk}-A_{kj}}.
$$
Therefore, (\ref{braket2}) holds true if and only if $A$
is symmetric. 
In this case, one can solve for an explicit solution for an
arbitrary starting point $(\overline x,\widetilde x,s)$.
The diffeomorphism $\widehat
\Lambda(\xi,z,t) = \mat{\Lambda_t(\xi,z)}{t}$ is solved
$\Lambda_t(\xi,z) =  \mat{\xi}{g}$ with $g(\xi,z,t) = z-\frac{1}{2}\xi^\top A(t)\xi$
following the proof of Proposition \ref{diff1} in the Appendix (see 
\cite{Kouritzin16} for details on a more involved example).
Hence, $\pi=I_d$, $\widehat\sigma= \left[ \begin{array}{cc}I_d \\0\end{array}\right]$,
$\overline\kappa$ does not exist so $G(t)=I_d$ and	
\(
[\nabla\Lambda_t]^{-1}=
\left[
\begin{array}{cc}	
	I_d &0\\ \xi^\top A(t)&1
\end{array}	
\right].
\)
Now, we can take any functions $\overline\theta\in\mathbb R^d$, $\overline\beta\in\mathbb R^{d\times d}$, $\widetilde h\in\mathbb R$
satisfying the differentiability conditions in Definition \ref{def3}
and let $\widetilde X_t,\ U_{s,t}\widetilde X,\ \overline c_s(t)$ satisfy:
\begin{eqnarray*}
\frac{d}{dt} \widetilde X_t &\!\!= &\!\!\widetilde h(\widetilde X_t,t)\ \text{s.t.\ }\widetilde X_s=\widetilde x-\frac{1}{2}\overline x^\top A(s)\overline x
\\
\frac{d}{du}U_{s,u}\widetilde X&\!\!=&\!\!-(U_{s,u}\widetilde X)\,\overline \beta(\widetilde X_u,u)\  \text{s.t.\ }U_{s,s}\widetilde X=I_d	
\\
\overline c_s(t) &\!\!=&\!\! U^{-1}_{s,t}\left\{\overline x +\disp \int_0^t
U_{s,u}\overline \theta(\widetilde X_u,u)du \right\}.
\end{eqnarray*}
From Theorem \ref{r3} and (\ref{bh}), drift $b$ must be of the (quadratic) form
$$
b(\xi,z,t)\! =\! \mat{\overline\theta(\widetilde X_t,t)-\overline\beta(\widetilde X_t,t)\xi}
{\widetilde h(\widetilde X_t,t) +
\xi^\top A(t)\overline\theta(\widetilde X_t,t)-\xi^\top A(t) \overline\beta(\widetilde X_t,t)\xi +\frac12 \xi^\top \frac{d}{dt}A(t)\xi +\frac{1}{2}{\rm Tr}\{A(t)\} }
$$
for some $\overline\theta$, $\overline\beta$, $\widetilde h$.
Finally, the corresponding $\phi$ is given by
$$
\phi(y,t) = \mat{\overline c_s(t)+(U^{-1}_{s,t}\widetilde X)y}{\widetilde X_t
+\frac{1}{2}(\overline c_s(t)+(U^{-1}_{s,t}\widetilde X)y)^\top
A(t)(\overline c_s(t)+(U^{-1}_{s,t}\widetilde X)y)}.
$$
\end{ex}


\section{Comparison with the works of Yamato and Kunita} \label{kun}
Now, we compare our existence results to those appearing in
\citet{Yamato:1979}  and \citet{Kunita:1984}. In Section III.3 of
Kunita's treatise, he considers representations of time-homogeneous
Fisk-Stratonovich equations
\begin{equation} \label{N2.6a}
dX^{x}_t = h(X^{x}_t)dt+\sigma(X^{x}_t)\bullet dW_t
\end{equation}
in terms of the flows generated by the vector
fields
\begin{equation}\label{VectField}
{\mathfrak X}_0(y)= \disp \sum_{i=1}^ph_i(y)\frac \partial
{\partial y_i} \text{ and }{\mathfrak X}_k(y)= \disp \sum_{i=1}^p \sigma
_{ik}(y)\frac \partial {\partial y_i}, k=1,...,d,
\end{equation}
under
conditions imposed on the Lie algebra $L_0({\mathfrak X}_0,
{\mathfrak X}_1, \ldots, {\mathfrak X}_d)$ generated by
${\mathfrak X}_k$, $0\le k\le d$. In the special case where these
vector fields commute, i.e. the Lie bracket $[{\mathfrak
X}_k,{\mathfrak X}_j]=0 $ for each $j,k=0,...,d$, and the
coefficients $h_i$, $\sigma _{ik}$ are respectively in
$C^3_\alpha$, $C^4_\alpha$ (the locally four times
continuously differentiable functions whose fourth derivative is $\alpha $%
-H\"older continuous), his work gives rise to the composition formula
\begin{eqnarray}
\left( X_t^{x}\right) _i &=& Exp\left( t{\mathfrak X}_0\right) \circ Exp\left(
W_t^1{\mathfrak X}_1\right) \circ \cdots \circ Exp\left( W_t^d{\mathfrak X}_d\right)
\circ \chi_i(x),  \label{Kunita} \\
\  &=&\phi _i(W_t,t)  \nonumber
\end{eqnarray}
locally. Here, $\chi _i$ is the function taking $x$ to its $i^{th}$
component and $Exp\left( u{\mathfrak X}_k\right) $ is the one parameter group of
transformations generated by vector field ${\mathfrak X}_k$, i.e. the unique solution to
\begin{equation}
\frac d{du}(f\circ \varphi _u)={\mathfrak X}_kf(\varphi _u),\;\varphi _0=x\qquad \forall
f\in C^\infty .  \label{Kde}
\end{equation}
In fact, to use (\ref{Kunita}), one must solve (\ref{Kde}) for
$k=0,...,d$ and $f=\chi _i$, $i=1,...,d$. Kunita also goes beyond
commutability, even surpassing \citet{Yamato:1979}  in generality by
considering the situation where $L_0({\mathfrak X}_0,...,{\mathfrak
X}_d)$ is only solvable, but the expression replacing (\ref{Kunita})
necessarily becomes more unwieldy.

Our characterization of $\phi $ provided by Theorem \ref{r3} provides an
alternative to (\ref{Kunita}) that is more amenable to direct
calculation. 
Corollary \ref{corr3} (to follow) supplies a converse to
(\ref{Kunita}) in the sense that if $X_t^{x,s}$ were to have such a
functional representation $\phi^{x,s}(W_t,t)$ in terms of Brownian
motions only, then the vector fields must commute. This was
previously established in Theorem 4.1 of \citet{Yamato:1979}  under
$C^\infty$ conditions on both $\phi$ and the coefficients. 

The other advantages of our representations over Kunita's results
are:
\begin{itemize}
\item We allow time dependent vector fields.
\item We decrease the regularity assumptions by imposing weaker
differentiability
on $h$ and on $\sigma$ when $r$ is small. The looser
regularity on the
coefficients requires eschewing Fisk-Stratonovich equations in
favour of It\^o processes.
\item We remove the nilpotency assumptions (for our representations).
\end{itemize}

To validate the final claim, we take $p=2$, $d=1$,
\[
{\mathfrak
X}_0 = \{\overline \theta(x_2)-B(x_2)x_1\} \dx{1} + \widetilde \theta(x_2)\dx{2},
\]
and ${\mathfrak X}_1 = \dx{1}$. Then $[{\mathfrak X}_0,{\mathfrak
X}_1] = B\dx{1}$. Moreover, if $ {\mathfrak X}_k = [{\mathfrak
X}_0,{\mathfrak X}_{k-1}]$, $k\ge 2$, then  $ {\mathfrak X}_k =
a_k(x_2)\dx{1}$, where $a_{k+1} = \widetilde \theta (\dx{2} a_k)+ Ba_k$,
$k\ge 1$ and $a_1 =1$.
In general, the $a_k$'s will not vanish
and thereby the Lie algebra contains an infinite number of
linearly independent vector fields. This algebra is solvable but
is not nilpotent.

\medskip
Using Theorem \ref{r1}, we can also give the converse to Kunita's result,
Example III.3.5 in \citet{Kunita:1984}, that is valid under the mild
regularity on $b,\sigma,h$ given at the beginning of the section.

\bigskip
\begin{cor}\label{corr3}
Suppose that there exists a domain $\widetilde{D}$ such that the
coefficients $\sigma$ and $h$ are time-homogeneous and
Fisk-Stratonovich acceptable on
$\tilde D_T = \tilde D\times (0,T)$.
Further, assume that the solution to the Fisk-Stratonovich equation
(\ref{N2.6a}) has a unique local solution
$$
\left( X_t^{x}\right) _i=Exp\left( t{\mathfrak X}_0\right) \circ Exp\left( W_t^1{\mathfrak X}_1\right)
\circ \cdots \circ Exp\left( W_t^d{\mathfrak X}_d\right) \circ \chi _i(x)
$$
on $0\leq t<\tau _x$ for some positive stopping time $\tau _x$ and
each $x\in \widetilde{D}$, where ${\mathfrak X}_k$, $k=0,1,\ldots,
d$ are the vector fields defined in (\ref{VectField}).
Then,
$$
[{\mathfrak X}_k,{\mathfrak X}_j]=0    \mbox{ on }\widetilde{D}\mbox{ for each }j,k=0,\ldots,d.
$$
\end{cor}

\begin{proof}
We find that $X^{x}_t=\phi(Y_t,t)$ with $U_{s,t}=I$ so it follows
from Theorem \ref{r1} that $\sigma A=0$.  
The condition $[{\mathcal
X}_k,{\mathcal X}_j]=0$ then follows from (\ref{braket2},\ref{h}).
\end{proof}


\section{Proofs of the main results}\label{proofs}

We note that $b,\sigma$ are Lipschitz on
any compact, convex subset of $D_T$ by our $C^1$-conditions and use
the proof  of \citet{Kunita:1984}[Theorem II.5.2] for
uniqueness of (strong) local solutions to the SDE until they leave such a compact subset.

\subsection{Proof of Theorem \ref{r1} a) is equivalent to b).}\label{Prop1proof}
\begin{proof}
Using (\ref{Ydefn}) and It\^o's formula for $X_t=\phi(Y_t,t)$, one finds that for any
$1\le i\le p$,
\begin{eqnarray}
\!\!&\!\!\!&\!\!\!d(X_t)_i =\sum_{m=1}^d\sum_{j=1}^d\partial
_{y_m}\phi_i(Y_t,t)(U_{s,t}\phi)_{mj}dW_t^j\label{ctscoef}\\
&+&\!\!\!\!\left[ \partial _t\phi_i(Y_t,t)+ \nonumber
\frac{1}{2}\sum_{j=1}^d\sum_{k=1}^d\partial _{y_j}\partial
_{y_k}\phi_i(Y_t,t)(U_{s,t}\phi\ (U_{s,t}\phi)^{\top })_{jk}\right]\! dt.
\end{eqnarray}
Now, starting with b) implies a) and using (\ref{grad},\ref{dtphi}) on (\ref{ctscoef}),
we find
\begin{eqnarray}
\!\!d(X_t)_i \!\!\!&=&\!\!\!\sigma_i(\phi(Y_t,t),t) dW_t+h_i(\phi(Y_t,t),t)dt\label{ctscoef1}\\
 &+&\!\! \nonumber
\frac{1}{2}\sum_{j=1}^d\sum_{k=1}^d\partial _{y_j}\partial
_{y_k}\phi_i(Y_t,t)(U_{s,t}\phi\ (U_{s,t}\phi)^{\top } )_{jk} dt.
\end{eqnarray}
Moreover,
$$
\dy{m} \{\sigma_{ij}(\phi,t)\} = \sum_{n=1}^p\{\partial _{\varphi_n}\sigma
_{ij}\}(\phi,t
)\dy{m}\phi_n
$$ 
and if (\ref{grad}) is true, one obtains
$$
\dy{m} \{\sigma_{ij}(\phi,t)\} =\sum_{l=1}^d\dy{m}\dy{l}\phi_i\;(U_{s,t}\phi )_{lj}.
$$
Abbreviating notation $U_{m k}(\phi,t)=(U_{s,t}\phi) _{mk}$,
multiplying the last two equalities by $U_{m k}$, summing over $m$ and using (\ref{grad}) again,
 one finds that
\begin{equation}\label{iden}
\sum_{n=1}^p\{\partial _{\varphi_n}\sigma _{ij}\}(\phi,t )\sigma
_{nk}(\phi,t )=\sum_{m=1}^d\sum_{l=1}^d \partial _{y_m}\partial
_{y_l}\phi_i\;U_{lj}(\phi,t)U_{mk}(\phi,t),
\end{equation}
and, taking $k=j$ and summing over $j$, one has that
\begin{equation}\label{corrreduce}
\!\sum_{j=1}^d \{\nabla_\varphi \sigma_j\}(\phi,t)\sigma_j (\phi,t) =\!
\sum_{l=1}^d\sum_{m=1}^d  \;(U(\phi,t)U^{\top }(\phi,t))_{lm}  \dy{m}\dy{l} \phi .
\end{equation}
Therefore, if (\ref{grad},\ref{dtphi},\ref{initphi}) are satisfied, then clearly $X_t$ is a local strong solution to
(\ref{SDE4}) by (\ref{bh}).
Moreover, letting $t\searrow s$, we find by (\ref{grad},\ref{dtphi},\ref{initphi}) that
\[
\sigma(x,s)=\nabla_y \phi^{x,s}(0,s)\ \text{ and } h(x,s)=\partial_t \phi^{x,s}(0,s)
\]
so $\sigma,h\in C^1$ by the last part of Definition \ref{RepPair}.

To show a) implies b), we suppose $X_t$ is a strong solution to
(\ref{SDE4}) on $(s,\tau^{x,s})$.
Then, since continuous finite-variation martingales are
constant, the (continuous) It\^o process 
$\phi(Y_t,t)$ from (\ref{ctscoef}) matches (\ref{SDE4}) if
and only if
\begin{equation} \label{sigma0}
\sigma _{ij}(\phi,t )=\sum_{m=1}^d\partial _{y_m}\phi _i\;(U_{s,t}\phi)_{mj}\;\forall 1\le i\leq
p,\;1\le j\le d,
\end{equation}
and
\begin{equation} \label{b0}
b_i(\phi,t )=\partial _t\phi_i+{\frac
12}\sum_{j=1}^d\sum_{k=1}^d\partial _{y_j}\partial
_{y_k}\phi_i\;(U_{s,t}\phi (U_{s,t}\phi)^{\top })_{jk}\;\forall 1\le i\le p
\end{equation}
for all $t\in(s,\tau^{x,s})$.
Rewriting (\ref{sigma0}) in matrix form, one finds
\begin{equation}
\sigma (\phi(Y_t,t),t )=\{\nabla _y\phi(Y_t,t) \}U_{s,t}\phi \label{crosscross},
\end{equation}
and (\ref{grad}) is true.
Now, we can use (\ref{corrreduce}) (which was just shown to be a consequence of (\ref{grad}))
to find (\ref{b0}) is equivalent to
\begin{equation} \label{N4.7a}
\partial _t\phi =b(\phi,t )-{\frac 12}\sum_{k=1}^d\{\nabla_x \sigma _k\}(\phi,t
)\;\sigma _k(\phi,t )=h(\phi,t ),
\end{equation}
using (\ref{bh}). 
Now, (\ref{dtphi}) follows by continuity and (\ref{Rdefn}).
Letting $t\searrow s$ in (\ref{crosscross}) and (\ref{N4.7a}), one finds
\[
	\sigma(x,s)=\nabla_y \phi^{x,s}(0,s)\ \text{ and } h(x,s)=\partial_t \phi^{x,s}(0,s)
\]
so the last part of Definition \ref{RepPair}
follows from the $C^1$ property of $h,\sigma$.
\end{proof}

\subsection{Proof of Theorem \ref{r1} b) is equivalent to c).}\label{Th1proof}
{\bf Idea:}
Below we show that the existence of a representation without the commutator conditions leads
to a contradiction and the commutator conditions yield a representation.
\begin{proof}
By exactness of differential 1-forms, the existence of our
function $\phi^{x,s}$ satisfying ((\ref{grad}), (\ref{dtphi}) and
(\ref{initphi})) is equivalent to the following two conditions:
\begin{equation}\label{exact1}
	\partial_{y_j}\{\sigma(\phi,t)(U^{-1}_{s,t}\phi)_k \}=
\partial_{y_k}\{\sigma(\phi,t)(U_{s,t}^{-1}\phi)_j\}
\end{equation}
and
\begin{equation}\label{exact2}
\frac{d}{dt}\{\sigma(\phi,t)(U_{s,t}^{-1}\phi)_k \}=
\partial_{y_k}h(\phi,t).
\end{equation}
We show (\ref{exact1}) and (\ref{exact2}) for all starting points $(x,s)$ are equivalent to (\ref{braket2}) and (\ref{h}) respectively.\\
{\bf Step 1:} Show that (\ref{braket2}) implies (\ref{exact1}) (under (\ref{grad})).\\
It follows by (\ref{grad}) and $\mathcal C_2$, $\mathcal C_3$ that
\begin{eqnarray}\label{N4.7b}
\partial_{y_j}\!\!\!\!\!\!\!&\!\!\!\!\!\!\!\!\!\!\!\!&\!\!\!\!\!\!\!\!\{\sigma(\phi,t)(U_{s,t}^{-1}\phi)_k \}\\
&=&
\sum_m \{\partial_{y_j} \sigma_m(\phi,t)\}(U_{s,t}^{-1}\phi)_{mk}\ \ \nonumber\\
&=&\sum_m \nabla_\phi \sigma_m(\phi,t)\sigma(\phi,t)(U_{s,t}^{-1}\phi)_j (U_{s,t}^{-1}\phi)_{mk}\ \ \nonumber\\
&=&\sum_m\sum_n \nabla_\phi \sigma_m(\phi,t)\sigma_n(\phi,t)(U_{s,t}^{-1}\phi)_{nj} (U_{s,t}^{-1}\phi)_{mk}\ \ \nonumber
\end{eqnarray}
and similarly
\begin{eqnarray}\label{Nov25}
\partial_{y_k}\!\!\!\!\!\!\!&\!\!\!\!\!\!\!\!\!\!\!\!&\!\!\!\!\!\!\!\!\{\sigma(\phi,t)(U_{s,t}^{-1}\phi)_j \}\\
&=&\sum_n \sum_m\nabla_\phi \sigma_n(\phi,t)\sigma_m(\phi,t)(U_{s,t}^{-1}\phi)_{mk} (U_{s,t}^{-1}\phi)_{nj}\nonumber
\end{eqnarray}
Hence, (\ref{exact1}) holds when (\ref{braket2}) holds.\\
{\bf Step 2:} Show that (\ref{exact1}) implies (\ref{braket2}) (under (\ref{grad})).\\
Letting $t\searrow s$ in (\ref{N4.7b}) and (\ref{Nov25}), one finds by (\ref{exact1})
that for all $1\le j,k \le d$,
\begin{eqnarray}
&&\sum_m\sum_n \nabla_x \sigma_m(x,s)\sigma_n(x,s)(U_{s,s}^{-1}\phi)_{nj} (U_{s,s}^{-1}\phi)_{mk}\\
&=&\lim_{t\searrow s}\partial_{y_j}\{\sigma(\phi,t)(U_{s,t}^{-1}\phi)_k \}\ \ \nonumber\\
&=&\lim_{t\searrow s}\partial_{y_k}\{\sigma(\phi,t)(U_{s,t}^{-1}\phi)_j \}\ \ \nonumber\\
&=&\sum_m\sum_n\nabla_x \sigma_n(x,s) \sigma_m(x,s)(U_{s,s}^{-1}\phi)_{nj} (U_{s,s}^{-1}\phi)_{mk}.\nonumber
\end{eqnarray}
However, $U_{s,s}^{-1}\phi=I$ 
so we have that
$$
(\nabla_x \sigma_q)(x,s)\sigma_p (x,s)= (\nabla_x
\sigma_p)(x,s)\sigma_q(x,s).
$$
Hence, (\ref{braket2})
holds when (\ref{exact1}) does.\\
{\bf Step 3:} Show that (\ref{exact2}) implies (\ref{h}) (under (\ref{grad},\ref{dtphi})).\\
One gets by (\ref{grad}) that
\begin{equation}
\partial_{y_k}h(\phi,t) = \nabla_\phi h(\phi,t)\partial_{y_k}\phi(y,t)=\nabla_\phi h(\phi,t)\sigma(\phi,t)(U_{s,t}^{-1}\phi)_k
\end{equation}
and by (\ref{exact2}), (\ref{dtphi}) that
\begin{eqnarray}
\partial_{y_k}h(\phi,t)\!\! &=&\!\! \frac{d}{dt}\{\sigma(\phi,t)(U_{s,t}^{-1}\phi)_k \}\\
&=& \!\!\sum_m\nabla_\phi \sigma_m(\phi,t)h(\phi,t)(U_{s,t}^{-1}\phi)_{mk}+\partial_{t}\sigma(\phi,t)(U_{s,t}^{-1}\phi)_k\nonumber\\
&-&\!\!\sigma(\phi,t)U_{s,t}^{-1}\phi\sum_m\frac{d}{dt}(U_{s,t}\phi)_m(U_{s,t}^{-1}\phi)_{mk}.\nonumber
\end{eqnarray}
Combining these equations, multiplying by $(U_{s,t}\phi)_{kn}$ and summing, we get
\begin{eqnarray}
\nabla_\phi h(\phi,t)\sigma_n(\phi,t)&=&\nabla_\phi
\sigma_n(\phi,t)h(\phi,t)+\partial_{t}\sigma_n(\phi,t)\\
&-&\sigma(\phi,t)U^{-1}_{s,t}\phi^{x,s}\frac{d}{dt}(U_{s,t}\phi^{x,s})_n\nonumber
\end{eqnarray}
so, letting $t\searrow s$ and using (\ref{SecondA},\ref{initphi}), one arrives at (\ref{h}).\\
{\bf Step 4:} Show that (\ref{h}) implies (\ref{exact2}) (under (\ref{grad},\ref{dtphi})).\\
Using (\ref{dtphi}) and (\ref{Ugroup}), we get that
\begin{eqnarray*}
&&\frac{d}{dt}\{\sigma(\phi,t)(U^{-1}_{s,t}\phi)_k\}\\
&\!\!=&\!\! \sum_n\Big[\{\nabla_\phi \sigma_n(\phi,t)\}h(\phi,t)
+\{\partial_t\sigma_n(\phi,t)\}\Big](U^{-1}_{s,t}\phi)_{nk}\\
&\!\!&\!\!\quad -\sum_n\sigma(\phi,t)U^{-1}_{s,t}\phi\left(\frac{d}{dt}(U_{s,t}\phi)_n\right)(U^{-1}_{s,t}\phi)_{nk}\\
&\!\!=&\!\! \sum_n\Big[\{\nabla_\phi \sigma_n(\phi,t)\}h(\phi,t)
+\partial_t\sigma_n(\phi,t)\\
&\!\!&\!\!-\sigma(\phi,t)\left(\frac{d}{dt}U_{u,t}\phi^{\phi_u,u}\right)_n\big|_{u=t}\Big]
(U^{-1}_{s,t}\phi)_{nk},
\end{eqnarray*}
where $\phi_u$ is short for $\phi^{x,s}(y_u,u)$.
Hence, by (\ref{SecondA}), (\ref{h}) applied at $\varphi=\phi$ and (\ref{grad})
\begin{eqnarray*}
\frac{d}{dt}\{\sigma(\phi,t)(U^{-1}_{s,t}\phi)_k\}
&\!\!=&\!\!\{\nabla_\phi h(\phi,t)\} \sigma(\phi,t) (U^{-1}_{s,t}\phi)_k \\
&\!\!=&\!\!\partial_{y_k}h(\phi,t)
\end{eqnarray*}
and we have (\ref{exact2}).
\end{proof}

\subsection{Proof of Proposition \ref{diff1}}
Our methods
are motivated in part by \citet{Brickell/Clark:1970}[Propositions 8.3.2 and 11.5.2].\\
We let
\[
	(q,D^2_T) = \left\{\begin{array}{ll}
			(p+1,D\times (-T,T)) & \mbox{if $\sigma$ or $h$ depend on $t$}\\
			(p,D) & \mbox{otherwise}
\end{array} \right.,
\]
take $\sigma_{p+1}=0$ if $q>p$,
set $\partial_t\sigma(x,t)=\partial_t\sigma(x,0)$,
$\partial_{x_i}\sigma(x,t)=\partial_{x_i}\sigma(x,0)$ for $t<0$, $i=1,2,...,q$
and use exactness of the corresponding $1$-form
to extend $\sigma$ uniquely to $D^2_T$ such that
$\sigma\in C^1(D^2_T;\mathbb R^{q\times d})$.
By reducing $T>0$ if necessary, we can find a permuation $\pi$ such that the first
$r$ columns of $\sigma^\pi=\sigma\pi$ are linearly independent on $D_T^2$.

\begin{proof}
Fix $\widehat x=(\widehat x_1,...,\widehat x_q)\in D_T$.
The $C^1$-diffeomorphism $\Lambda$ will have form:
\begin{eqnarray}
	\!\!\Lambda&\!\!=&\!\!\Lambda^{r,1},\ \ \text{where }\ \Lambda^{i,1}=\Lambda^{i}\circ\Lambda^{i-1}\circ\cdots\circ\Lambda^{2}\circ\Lambda^{1},\\\label{DiffForm}
\Lambda^{i}&\!\!=&\!\!\sum_{j=1}^{i-1}x_je_j+\left[\begin{array}{c}H^i(x_i,...,x_q)\\L^i(x_i,...,x_q)\end{array} \right]\ \ \text{and }\ H^i(x_i,...,x_q)\in\mathbb R^{i-1},\
\end{eqnarray}
and $\widehat \sigma_i$ will be defined as
$\widehat \sigma_i\doteq\{(\nabla_\varphi \Lambda)\sigma^\pi_i\}\circ \Lambda^{-1}$.
Here, 
$\Lambda^i$ is a $C^1$-diffeomorphism  
on a neighborhood $O^{\widehat x^{i-1}}$ of $\widehat x^{i-1}=
\Lambda^{i-1,1}(\widehat x)$
so $\Lambda:O^{\widehat x}\rightarrow \mathbb R^q$.

To construct $\Lambda^i$ recursively starting with $\Lambda^1$, we 
suppose $\widehat \sigma_j=e_j$ for $j<i$ and
\begin{equation}\label{alphadef}
\alpha_i\doteq\{\nabla\Lambda^{i-1,1}\sigma^\pi_i\}\circ(\Lambda^{i-1,1})^{-1}
\end{equation}
does not depend upon $x_1,...,x_{i-1}$, which are certainly true when $i=1$.
Moreover, without loss of generatlity, we assume the $i^{th}$ component of $\alpha_i$ satisfies 
$\alpha_{i,i}\ne 0$ (or else we change $\pi$ by permuting columns $i,...,d$
of $\sigma^\pi$).
Set $\psi^i(x) = \theta(x_i -\widehat x_{i}^{i-1}; x_1,...,x_{i-1},\widehat x_i^{i-1},x_{i+1},...,x_q)$,
where $\theta$ satisfies
$\frac{d}{dt}\theta(t;x)=\alpha_i(\theta(t;x))$, $\theta(0;x)=x$
for $t\in I^x$, an open interval containing $0$, and $x$
in a neighborhood containing $\widehat x^{i-1}$.
Then, $\dx{i} \psi^i = \alpha_i(\psi^i)$. 
For $j\ne i$, we have $\partial_{x_j}\psi^i(x)=\partial_{x_j}\theta(x_i -\widehat x_{i}^{i-1};x_1,...,x_{i-1},\widehat x_i^{i-1},x_{i+1},...,x_q)$ and
$$
\partial_t\partial_{x_j}\theta(t;x)=\partial_{x_j}\alpha_i(\theta(t;x))\ \text{  s.t. }\ \partial_{x_j}\theta(0,x)=e_j
$$
so $\nabla \psi^i(\widehat x^{i-1})$ has determinant
$\alpha_{i,i}(\widehat x^{i-1}) \neq 0$.
Thus, $\psi^i$ has inverse $\Lambda^i\in C^2(O^{\widehat x^{i-1}},\mathbb R^{q})$ 
and $\n \Lambda^i = [\n \psi^i]^{-1}(\Lambda^i)$ on neighborhood
$O^{\widehat x^{i-1}}$ 
of $\widehat x^{i-1}$ by the Inverse Function Theorem.
Hence, $\nabla \Lambda^i((\Lambda^i)^{-1}) \nabla \psi^i = I$ and 
\begin{equation}\label{OneCol}
\widehat \sigma_i=\{ \nabla \Lambda^i\alpha_i \}(\Lambda^i)^{-1} =e_i\in\mathbb R^{q}.
\end{equation}
Moreover, $\Lambda^i$ has the form (\ref{DiffForm}) if $\psi^i$ has similar form.  
$\psi^i$ has this form by its definition 
as well as the facts $\alpha_i$ is locally Lipschitz and does not depend upon $x_1,...,x_{i-1}$.
Next,
\begin{equation}\label{TransCom}
(\nabla \widehat \sigma_j) \widehat \sigma_k - (\nabla \widehat \sigma_k)\widehat \sigma_j = (\nabla \sigma^\pi_j)
\sigma^\pi_k -(\nabla \sigma^\pi_k)\sigma^\pi_j=0\ \ \forall\ \ 1\le k , j\le d
\end{equation}
by Lemma \ref{dc}. 
Now, since $\widehat \sigma_k=e_k\in\mathbb R^q$ for $1\le k\le i$, (\ref{TransCom}) implies
$$
(\nabla \widehat \sigma_j) e_k = (\nabla \widehat \sigma_j) e_k - (\nabla e_k)\widehat \sigma_j = 0\ \ \forall\ 1\le k \leq i<j 
$$
on a neighborhood $O$ of $\widehat x$.
Therefore, $\widehat \sigma_j$ and (by a similar argument) $\alpha_{i+1}$ can not 
depend upon $x_1, \ldots, x_i$ so we can take $i=r$ by induction and
$$
\widehat \sigma =
\{(\nabla \Lambda) \sigma^\pi\}\circ \Lambda^{-1} =
\left(\begin{array}{cc} I_r & \overline \kappa\\ 0 & \widetilde \kappa \end{array}\right)
\in \mathbb R^{q\times d}  \mbox{ on } \Lambda(O\cap D_T),
$$
where $\overline \kappa \in {\mathbb R}^{r\times (d-r)}$  and
$\widetilde \kappa \in {\mathbb R}^{(q-r)\times (d-r)}$ do
not depend on the variables
$x_1, \ldots, x_r$.
Since $\widehat \sigma$ has also rank $r$, it follows that $\widetilde \kappa =0$.
\end{proof}

\bibliographystyle{apalike}

\bibliography{sde6}
\bigskip

\end{document}